\documentclass[11pt, notitlepage]{article}
\usepackage{amssymb,amsmath,comment}
\catcode`\@=11 \@addtoreset{equation}{section}
\def\thesection{\arabic{section}}

\def\theequation{\thesection.\arabic{equation}}
\catcode`\@=12
\usepackage{colortbl}%
\usepackage{a4wide}

\newcommand{\fa} {\forall}
\newcommand{\ds} {\displaystyle}
\newcommand{\e}{\epsilon}

\newcommand{\al} {\alpha}
\newcommand{\ba} {\beta}
\newcommand{\de} {\delta}
\newcommand{\ga} {\gamma}

\newcommand{\Om} {\Omega}
\newcommand{\ra} {\rightarrow}

\newcommand{\De} {\Delta}
\newcommand{\la} {\lambda}
\newcommand{\La} {\Lambda}
\newcommand{\noi} {\noindent}
\newcommand{\na} {\nabla}

\newcommand{\oline} {\overline}
\newcommand{\mb} {\mathbb}
\newcommand{\mc} {\mathcal}
\newcommand{\lra} {\longrightarrow}
\newcommand{\ld} {\langle}
\newcommand{\rd} {\rangle}

\usepackage[all]{xy}
\catcode`\@=11
\def\theequation{\@arabic{\c@section}.\@arabic{\c@equation}}
\catcode`\@=12

\def\proof{\noindent{\textbf{Proof. }}}
\def\QED{\hfill {$\square$}\goodbreak \medskip}

\newtheorem{Theorem}{Theorem}[section]
\newtheorem{Lemma}[Theorem]{Lemma}
\newtheorem{Proposition}[Theorem]{Proposition}

\newtheorem{Definition}[Theorem]{Definition}

\begin{document}

\title
{$n$-Kirchhoff type equations with exponential nonlinearities}

\author{
{\bf   \; Sarika Goyal\footnote{email: sarika1.iitd@gmail.com}},\; {\bf Pawan Kumar Mishra\footnote{email: pawanmishra31284@gmail.com,}}\; and {\bf  K. Sreenadh\footnote{e-mail: sreenadh@gmail.com}}\\
{\small Department of Mathematics}, \\{\small Indian Institute of Technology Delhi}\\
{\small Hauz Khas}, {\small New Delhi-16, India}\\
 }

\date{}

\maketitle

\begin{abstract}

In this article, we study the existence of
non-negative solutions of the class of non-local problem of $n$-Kirchhoff type
$$  \left\{
\begin{array}{lr}
 \quad - m(\int_{\Om}|\na u|^n)\De_n u = f(x,u) \; \text{in}\;
\Om,\quad u =0\quad\text{on} \quad \partial \Om,
\end{array}
\right.
$$
where $\Om\subset \mb R^n$ is a bounded domain with smooth boundary, $n\geq 2$ and $f$ behaves like $e^{|u|^{\frac{n}{n-1}}}$ as $|u|\ra\infty$. Moreover, by minimization on the suitable subset of the Nehari manifold, we study the existence and multiplicity of solutions, when $f(x,t)$ is concave near $t=0$ and convex as $t\rightarrow \infty.$

\medskip

\noi \textbf{Key words:} Kirchhoff equation,
Trudinger-Moser embedding, sign-changing weight function.

\medskip

\noi \textit{2010 Mathematics Subject Classification:} 35J35, 35J60,
35J92

\end{abstract}

\bigskip
\vfill\eject

\section{Introduction}
\setcounter{equation}{0}
The aim of this article is to study the existence of positive solutions of following $n$-Kirchhoff type equation
$$ \mc{( M)}\quad \left\{
\begin{array}{rllll}
 -m(\int_{\Om} |\na u|^n)\Delta_n u&=f(x,u) \; \text{in}\;  \Omega,\\
  u&=0\; \text{on}\;  \partial \Omega,
\end{array}\right.$$
where $\Omega\subset \mb R^n$ is a bounded domain with smooth boundary,
$m:\mb R^+\rightarrow \mb R^+$ and $f:\Omega\times\mb R\rightarrow \mb R$ are continuous functions that satisfy some conditions which will be stated later on.\\

\noi We also study the existence of non-negative solutions of the following $n$-Kirchhoff problem
$$ (\mc P_{\la,M})\quad \left\{
\begin{array}{rllll}
 \quad  -m(\int_{\Om}|\na u|^n)\De_n u &= \la h(x)|u|^{q-1}u+ u|u|^{p}~ e^{|u|^{\ba}} \; \text{in}\;
\Om \\
  u &=0\quad \text{on} \quad \partial \Om,
\end{array}
\right.
$$
where $\Om\subset \mb R^n$ is a bounded domain with smooth boundary, $n\geq 2$, $0< q<n-1<2n-1 <
p+1$, $\ba\in \left(1,\frac{n}{n-1}\right]$ and $\la>0$. By
minimization on the suitable subset of the Nehari manifold we show the existence and multiplicity of solutions with respect to the parameter $\la$.\\

\noi The above problems are called non-local because of the presence of the term $m(\int_\Om |\na u|^n)$ which implies that the equations
in $(\mc M)$ and $(\mc P_{\la, M})$ are no longer a pointwise identity. This phenomenon causes some mathematical difficulties which makes the study
of such a class of problem interesting. Basically, the presence of $\int_\Om |\na u|^n$ as the coefficient of $\int_{\Om} |\na u|^{n-2}\na u \na \phi$ in the weak formulation makes the study of compactness of Palais-Smale sequences difficult. The study of elliptic equations with exponential growth nonlinearities are motivated by
the following Trudinger-Moser inequality  \cite{moser}, namely
\begin{Theorem} For $n\geq2$, $u\in W^{1,n}_{0}(\Om)$
\begin{align}\label{eqc002}
\sup_{\| u\| \leq 1}\int_{\Om} e^{\al |u|^{\frac{n}{n-1}}} dx
<\infty
\end{align}
if and only if ${\al}\leq {\al_n}$, where $\al_n = n
w_{n-1}^{\frac{1}{n-1}}$, $w_{n-1}=$ volume of $\mb S^{n-1}$.
\end{Theorem}
 The embedding $W_{0}^{1,n}(\Om) \ni u\longmapsto e^{|u|^{\ba}} \in
L^{1}(\Om)$ is compact for all $\ba\in\left(1,\frac{n}{n-1}\right)$
and is continuous for $\ba=\frac{n}{n-1}$. The non-compactness of
the embedding can be shown using a sequence of functions that are
truncations and dilations of fundamental solution of $-\De_n$ on
$W^{1,n}_{0}(\Om)$. The existence results for quasilinear problems
with exponential terms on bounded domains was initiated and studied
by Adimurthi \cite{A}.\\

\noi Starting from the pioneering works of Tarantello \cite{TA} and
Ambrosetti-Brezis-Cerami \cite{ABC}, a lot of work has been done to
address the multiplicity of positive solutions for semilinear and
quasilinear elliptic problems with positive nonlinearities.
Recently, many works are devoted to the study of these multiplicity
results with polynomial type nonlinearity with sign-changing weight
functions using the Nehari manifold and fibering map analysis (see
refs.\cite{TA,DP,WU,WU9,WU10,WUFI,COA, AE}). In \cite{CK}, authors studied the existence of
multiple positive solution of Kirchhoff type problem with convex-concave polynomial type nonlinearities
having subcritical growth by Nehari manifold and fibering map methods. In addition, the corresponding results
of the Kirchhoff type problem can be found in \cite{AG, AF, CX, BW, CH, FG, gs, HZ, XZ,LF} and references therein.

\noindent The boundary value problems involving Kirchhoff equations arise in several physical and
biological systems. These type of non-local problems were initially observed by Kirchhoff in 1883
 in the study of string or membrane vibrations to describe the transversal oscillations of a stretched string,
 particularly, taking into account the subsequent change in string length caused by oscillations.\\

\noindent In this paper, first we discuss the Adimurthi \cite{A}
type existence result for the $n$-Kirchhoff problem in $(\mc{M})$
with nonlinearity $f(x,u)$ that has  superlinear growth near zero
and exponential growth near $\infty$. To prove our result we follow
the approach as in \cite{gs}. In our case, the operator $-\De_n$ is
not linear, so we required to prove the pointwise convergence of
gradients of Palais-Smale sequences. Moreover due to Kirchhoff
operator we need the norm convergence of Palais-Smale sequence to
show that weak limit is a solution. We used concentration
compactness principle to show this convergence.
In the second part, we discuss the $n$-Kirchhoff problem in $(\mc{P}_{\la,M})$ with sign-changing and exponential type nonlinearity to obtain the multiplicity of solutions with respect to the parameter $\la$. We show the multiplicity result by extracting Palais-Smale sequences in the Nehari manifold. The results obtained here are some how expected but we show how the results arise out of nature of Nehari manifold.\\

\noi The paper is organized as follows: In section 2, we consider the critical problem with positive nonlinearity and prove Adimurthi's type \cite{A} existence result. In section 3, we study the problem with convex-concave sign-changing nonlinearity by
Nehari manifold approach and show the existence of two solutions that arise from the nature of the Nehari manifold.\\

\noi We shall throughout use the following notations: The norm on
$W^{1,n}_{0}(\Om)$ and $L^{p}(\Om)$ are denoted by $\|\cdot\|$,
$\|u\|_{p}$ respectively. The weak convergence is denoted by
$\rightharpoonup$ and $\ra$ denotes strong convergence.

\section{Existence of positive solutions with positive nonlinearity}
\setcounter{equation}{0}
\noi In this section, we prove the existence result for the problem
\[\mc{(M)} -m(\|u\|^n)\Delta_n u=f(x,u) \; \text{in}\;  \Omega,\quad\quad  u=0\; \text{on}\;  \partial \Omega,\]
where $\Omega\subset \mb R^n$ is a bounded domain with smooth boundary,
$m:\mb R^+\rightarrow \mb R^+$ and $f:\Omega\times\mb R\rightarrow \mb R$ are continuous functions that satisfy the following assumptions:
\begin{enumerate}
  \item[$(m1)$] There exists $m_0>0$ such that $m(t)\geq m_0$ for all $t\geq 0$ and
  \[ M(t+s)\geq M(t)+M(s)\;\text{for all}\; s,t\geq 0,\]
  where $M(t)= \int_{0}^{t} m(s)ds$, the primitive of $m$ so that $M(0)=0.$
  \item[$(m2)$] There exist constants $a_1$, $a_2>0$ and $t_0>0$ such that for some $\sigma\in\mb R$
  \[  m(t)\leq a_1 + a_2 t^{\sigma}, \; \text{for all}\; t\geq t_0.\]
  \item[$(m3)$] $\frac{m(t)}{t}$ is nonincreasing for $t>0$.
 \end{enumerate}
The condition $(m1)$ is valid whenever $m(0)=m_0$ and $m$ is
nondecreasing. A typical example of a function $m$ satisfying the
conditions $(m1)-(m3)$ is $m(t)=m_0+at^\al$, where $m_0>0,$ $a\geq
0$ and $\al>0$. Another example is $m(t)=1+\log(1+t)$ for $t\geq 0$.\\
 From $(m3)$, we can easily deduce that
\[ \frac{1}{n}M(t)-\frac{1}{\theta}m(t)t \;\text{ is nondecreasing for }\; t\ge 0\;\text{and}\; \theta\ge 2n.\]
In particular, one has
\begin{equation}\label{7a0}
\frac{1}{n}M(t)-\frac{1}{\theta}m(t)t\ge 0\; \text{for all}\; t\ge
0\;\text{and}\; \theta\ge 2n.
\end{equation}
The nonlinearity $f(x,t)=h(x,t) e^{|t|^{n/n-1}}$, where $h(x,t)$ satisfies
\begin{enumerate}
\item[$(f1)$] $h\in C^1(\overline{\Om}\times \mb R)$, $h(x,0)=0,$ for all $t\le 0$, $h(x,t)>0,$ \text{for all}
$t>0$ and $\lim_{t\ra 0} \frac{h(x,t)}{|t|^n}=0$.
\item[$(f2)$] For any $\e>0,$ $\ds \lim_{t\ra \infty}\sup_{x\in \overline{\Om}} h(x,t) e^{-\e |t|^{n/n-1} }=0$, $\ds\lim_{t\ra \infty}\inf_{x\in \overline{\Om}} h(x,t) e^{\e|t|^{n/n-1}}=\infty.$
\item[$(f3)$] There exist positive  constants $t_0$, $K_0>0$ such that
\[  F(x,t)\le K_0 f(x,t)\;\mbox{for all}\; (x,t)\in \Om\times[t_0,+\infty).\]
\item[$(f4)$] For each $\ds x\in \Omega, \frac{f(x,t)}{t^{2n-1}}$ is increasing for $t>0$ and $\ds\lim_{t\rightarrow 0^+} \frac{f(x,t)}{t^{2n-1}}=0,\;\text{uniformly in }\; x\in \Om.$
\item[$(f5)$] $\ds \lim_{t\rightarrow \infty} t h(x,t)=\infty.$
\end{enumerate}
\noi Assumption $(f3)$ implies that $\ds  F(x,t)\ge
F(x,t_0)e^{\frac{1}{K_0}(t-t_0)}$, for all $(x,t)\in \mb R^n \times
[t_0, \infty)$ which is a reasonable condition for function behaving
as $e^{\al_0 |t|^{n/n-1}}$ at $\infty.$ Moreover from $(f3)$ it
follows that for each $\theta>0, $ there exists $R_\theta>0$
satisfying
\begin{equation}\label{n7a0}
\theta F(x,t)\le t f(x,t)\; \text{for all}\; (x,t)\in \Om \times [R_\theta, \infty).
\end{equation}
We also have that condition $(f4)$ implies that for $\mu \in [0,2n-1)$,
\begin{equation}\label{n1}
\lim_{t\rightarrow 0^+} \frac{f(x,t)}{t^\mu}=0, \; \text{uniformly in }\; x\in \Om.
\end{equation}
\noi Generally, the main difficulty encountered in non-local Kirchhoff problems is the competition between the growths of $m$ and $f$. Here we generalize the result of \cite{gs} to the $n$-Kirchhoff equation.
\begin{Definition}
We say that $u\in W_0^{1,n}(\Omega)$ is a weak solution of $(\mc{M})$ if holds
\[m(\|u\|^n)\int_\Om |\nabla u|^{n-2}\nabla u\nabla \phi ~dx=\int_\Om f(x,u)\phi~ dx  \;\; \text{for all}\; \;\phi \in W_0^{1,n}(\Omega).\]
\end{Definition}
\noi The energy functional $J:W^{1,n}_{0}(\Om)\rightarrow \mb R$ corresponding to the problem $(\mc M)$ is defined as
\[ J(u)=\frac{1}{n}M(\|u\|^n)-\int_\Om F(x,u)~ dx.\]
Then the functional $J$ is Fr$\acute{e}$chet differentiable and the critical points are the weak solutions of $(\mc M)$.
We prove the following Theorem in this section:
\begin{Theorem}\label{thm711}
Suppose $(m1)-(m3)$ and $(f1)-(f3)$ are satisfied. Then, problem $(\mc {M})$ has a positive solution.
\end{Theorem}
\noi We prove this Theorem by mountain pass Lemma. In the next few Lemmas we studied the mountain pass structure and Palais-Smale sequence to the functional $J$.
\begin{Lemma}\label{lem7.1} Assume the conditions $(m1)$, $(f1)-(f3)$ hold. Then $J$ satisfies mountain-pass geometry around the $0$.
\end{Lemma}
\proof From the assumptions, $(f1)-(f3)$, for $\e>0$, $r>n$, there exists $C>0$ such that
\[|F(x,t)| \le \e |t|^n + C |t|^r e^{|t|^{n/n-1}},\;\; \text{for all}\; (x,t)\in \Omega \times \mb R.\]
Therefore, using Sobolev and H\"{o}lder inequalities, we get
\begin{align*}
\int_\Om F(x,u) dx &\le \e\int_\Om |u|^n dx +C \int_\Om |u|^r e^{|u|^{n/n-1}} dx \\
&\le \e C_1 \|u\|^n + C \|u\|_{2r}^{r} \left(\int_\Om e^{2\|u\|^{n/n-1} (\frac{u}{\|u\|})^{n/n-1}} \right)^{1/2}\\
&\le \e C_1 \|u\|^n + C_2 \|u\|^r
\end{align*}
\noi for $\|u\|<R_1$, where $R_1\leq \left(\frac{\al_n}{2}\right)^{\frac{n-1}{n}}$, thanks to Moser-Trudinger inequality \eqref{eqc002}.
Hence
\[J(u) \ge \|u\|^n \left(\frac{m_0}{n}-\e C_1 - C_2 \|u\|^{r-n}\right).\]
Since $r>n,$ we can choose $\e$, $0<R\leq R_1$ small such that $J(u)\ge \tau$ for some $\tau$ on $\|u\|=R$.\\
\noi Now by $\eqref{n7a0}$, for $\theta>\max\{n,n(\sigma+1)\}$, there exist $C_1$, $C_2>0$ such that
 \begin{equation}\label{n3}
 F(x,t)\geq C_1 t^{\theta}- C_2\;\mbox{for all}\; (x,t)\in \Om\times[0,+\infty)
 \end{equation}
 and for all $t\geq t_0$ condition $(m2)$ implies that
\begin{equation}\label{n2}
M(t)\leq\left\{
 \begin{array}{lr}
 a_0+a_1t+\frac{a_2}{\sigma+1} t^{\sigma+1},\; \mbox{if}\; \sigma\ne -1,\\
 b_0+ a_1 t+a_2 \ln t\quad\quad\mbox{if}\; \sigma=-1,
 \end{array}
 \right.
\end{equation}
where $a_0= M(t_0) -a_1t_0-a_2 t_0^{\sigma+1}/(\sigma+1)$ and $b_0= M(t_0) -a_1 t_0-a_2\ln t_0$. Now, choose a function $\phi_0 \in W^{1,n}_{0}(\Om)$ with $\phi_0\geq 0$ and $\|\phi_0\|=1$. Then from \eqref{n3} and \eqref{n2}, for all $t\geq t_0$, we obtain
\begin{equation*}
J(t\phi_0)\leq \left\{
\begin{array}{lr}
\frac{a_0}{n}+\frac{a_1}{n}t^n+\frac{a_2}{n\sigma+n} t^{n\sigma+n}- C_1 t^{\theta}\|\phi_0\|^{\theta}_{\theta}+C_2 |\Om|,\; \mbox{if}\; \sigma\ne -1,\\
\frac{ b_0}{n}+ \frac{a_1}{n} t^n +\frac{a_2}{n} \ln t - C_1 t^{\theta}\|\phi_0\|^{\theta}_{\theta}+C_2 |\Om|\;\;\;\quad\quad\mbox{if}\; \sigma=-1,
\end{array}
\right.
\end{equation*}
from which we conclude that $J(t u_0)\ra -\infty$ as $t\ra +\infty$ provided that $\theta>\max\{n, n\sigma+n\}$.
Therefore, $J$ satisfies mountain-pass geometry near $0$.\QED

\begin{Lemma}\label{lem712}
Every Palais-Smale sequence of $J$ is bounded in $W^{1,n}_{0}(\Om)$.
\end{Lemma}
{\proof} Let $\{u_k\}\subset W^{1,n}_{0}(\Om)$ be a Palais-Smale sequence for $J$ at level $c$, that is
\begin{equation}\label{n7a1}
\frac{1}{n}M(\|u_k\|^n)-\int_\Om F(x,u_k) \rightarrow c
\end{equation}
and for all $\phi \in W^{1,n}_{0}(\Om)$
\begin{equation}\label{n7a2}
 \left|-m(\|u_k\|^n)\int_\Om |\na u_k|^{n-2} \na u_k \na \phi dx -\int_\Om f(x,u_k) \phi dx\right| \le \e_k \|\phi\|
\end{equation}
where $\e_k\rightarrow 0$ as $k\rightarrow \infty.$
From \eqref{7a0}, \eqref{n7a0}, \eqref{n7a1} and \eqref{n7a2}, we obtain
 \begin{align*}
C+ \|u_k\|&\ge \frac{1}{n}M(\|u_k\|^n)-\frac{1}{\theta} m(\|u_k\|^n) \|u_k\|^n \\
&\quad\quad-\int_\Om \left(F(x,u_k)-\frac{1}{\theta}f(x,u_k)u_k\right)\\
&\geq \left(\frac{1}{2n}-\frac{1}{\theta}\right) m(\|u_k\|^n) \|u_k\|^{n}.
\end{align*}
From this and taking $\theta >2n$, we obtain the boundedness of the sequence.  \QED

\noi Let $\ds \Gamma=\{\gamma\in C([0,1],W_0^{1,n}(\Omega)):\gamma(0)=0,J(\gamma(1))<0\}$ and define the mountain-pass level
$\ds c_*=\inf_{\gamma\in \Gamma}\max_{t\in[0,1]}J(\gamma(t))$. Then we have,
\begin{Lemma}\label{lem7.2}
$\ds c_*<\frac{1}{n}M(\alpha_n^{n-1})$, where $\al_n= nw_{n-1}^{\frac{1}{n-1}}$, $w_{n-1}=$ volume of $n-1$ dimensional unit sphere in $\mb R^n$.
\end{Lemma}
\proof Let $\de_k>0$ be such that $\de_k \ra 0$ as $k\ra\infty$ and let $\phi_k(x)$ be the sequence of Moser functions defined by
\begin{equation}\label{moserfn}
 \phi_{k}(x) = \frac{1}{w_{n-1}^{\frac{1}{n}}}\left\{
\begin{array}{lr}
(\log k)^{\frac{n-1}{n}}& 0\leq \frac{|x|}{\de_k}\leq {\frac{1}{k}};\\
 \frac{\log{\frac{\de_k}{|x|}}}{(\log k)^{\frac{1}{n}}}& \frac{1}{k}\leq \frac{|x|}{\de_k}\leq
 1;\\
 0 & \frac{|x|}{\de_k}\geq 1,
\end{array}
\right.
\end{equation}
with support in $B_{\de_k}(0)\subseteq \mb R^n$. It can be easily seen
that $\|\na \phi_k\|_{n} =1$ for all $k$. Suppose the result is not true, i.e. $c_*\geq\frac{1}{n}M(\alpha_n^{n-1})$. Then for each $k$, there exists $t_k$ such that
\begin{equation}\label{7a1}
\sup_{t>0}J(t\phi_k)=J(t_k\phi_k)=\frac{1}{n}M(\|t_k\phi_k\|^n)-\int_{\Om} F(x,t_k\phi_k) \geq  \frac{1}{n}M(\alpha_n^{n-1}).
\end{equation}
From \eqref{7a1}, we see that $t_k$ is a bounded sequence as $J(t_k\phi_k)\rightarrow-\infty$ as $t_k\rightarrow\infty.$
Also using $M$ is monotone increasing and $F(x,t_k \phi_k)\geq 0$ in \eqref{7a1}, we obtain
\begin{equation}\label{7a2}
t_k^n\geq\alpha_n^{n-1}.
\end{equation}
Now since $t_k$ is a point of maximum for one dimensional map $t\mapsto J(t\phi_k),$ we have
$\frac{d}{dt}J(t\phi_k)|_{t=t_k}=0.$ From this it follows that
\begin{align}\label{new7a2}
 m(t_k^n\|\phi_k\|^n)t_k^n\|\phi_k\|^n &=\int_{\Om} f(x,t_k\phi_k)t_k\phi_k\geq \int_{B_\frac{\delta_k}{k}(0)} f(x,t_k\phi_k)t_k\phi_k\notag\\
 &=\phi_k(0)t_k h(x,t_k\phi_k(0))\frac{(\delta_k)^n}{k^n}.k^n \notag\\
&=\frac{t_k}{w_{n-1}^\frac{1}{n}}(\log k)^{\frac{n-1}{n}-\frac{1}{\al}}h(x,t_k\phi_k(0)).
\end{align}
\noi Now we choose $\de_{k}= (\log k)^{\frac{-1}{\al n}}$, with $\al
> \frac{n}{n-1}$. Then  $(f5)$ implies that the right hand side of
\eqref{new7a2} tends to $\infty$. Which is a contradiction as the
left side of \eqref{new7a2} is bounded. Hence
$c_*<\frac{1}{n}M(\alpha_n^{n-1})$.\QED

\noi In order to prove that a Palais-Smale sequence converges to a solution of problem ($\mc M$) we need  the following  convergence Lemma. We refer to Lemma 2.1 in \cite{fmr} for a proof.
\begin{Lemma}\label{lc1}
Let $\Om\subset \mb R^n$ be a bounded domain and $f:\overline{\Om}\times \mb R \ra \mb R$ a continuous function. Then for any sequence $\{u_k\}$ in
$L^{1}(\Om)$ such that
\[u_k\ra u\;\mbox{in}\; L^{1}(\Om),\quad f(x,u_k)\in L^{1}(\Om)\; \mbox{and}\; \int_{\Om}|f(x,u_k)u_k|\leq C,\]
we have up to a subsequence $f(x,u_k)\ra f(x,u)$ and $F(x,u_k)\ra F(x,u)$ strongly in $L^{1}(\Om)$.
\end{Lemma}

\noi Now we need the following Lemma, inspired by \cite{do6}, to show that weak limit of a Palais-Smale sequence is a weak solution of $(\mc M)$,

\begin{Lemma}\label{lem714}
For any Palais-Smale sequence $\{u_k\}$, there exists a subsequence still denoted by $\{u_k\}$ and $u\in W^{1,n}_{0}(\Om)$ such that $f(x,u_k) \rightarrow f(x,u)$ in $L^1(\Om)$ and $|\na u_k|^{n-2}\na u_k \rightharpoonup |\na u|^{n-2} \na u$ weakly in $(L^{n/n-1}(\Om))^n$.
\end{Lemma}
\proof From Lemma \ref{lem712}, we obtain that  $\{u_k\}$ is bounded in $W^{1,n}_{0}(\Om)$. Consequently,  up to a subsequence
$u_k\rightharpoonup u$ weakly in $W^{1,n}_{0}(\Om)$, $u_k\ra u$ strongly in $L^{q}(\Om)$ for all $q\in[1,\infty)$ and $
u_{k}(x)\ra u(x)\; \mbox{a.e in}\;\Om$. Then using the fact that $\{u_k\}$ is a bounded sequence together with \eqref{n7a2} and Lemma \ref{lc1}, we obtain $f(x,u_k)\ra f(x,u)$ in $L^{1}(\Om)$.

\noi Now to show that $|\na u_k|^{n-2}\na u_k \rightharpoonup |\na u|^{n-2} \na u$ weakly in $(L^{n/n-1}(\Om))^n$. First, we note that $\{|\na u_k|^{n-2}\na u_k\}$ is bounded in $L^{\frac{n}{n-1}}(\Om)$. Then, without loss of generality, we may assume that
\begin{eqnarray}
&|\na u_k|^n \lra \mu\;\mbox{in}\; D^{\prime}(\Om)\;\mbox{and}\;|\na u_k|^{n-2}\na u_k \rightharpoonup \nu \; \mbox{weakly in}\; L^{\frac{n}{n-1}}(\Om),
\end{eqnarray}
\noi where $\mu$ is a non-negative regular measure and $D^{\prime}(\Om)$ are the distributions on $\Om$.\\
\noi Let $\sigma>0$ and $\mc A_{\sigma}= \{x\in\overline{\Om}: \fa\; r>0, \mu(B_{r}(x)\cap \overline{\Om})\geq \sigma\}$. We claim that $A_{\sigma}$ is a finite set.
Suppose by contradiction that there exists a sequence of distinct points $(x_s)$ in $\mc A_{\sigma}$. Since for all $r>0$,
$\mu(B_{r}(x_s)\cap \overline{\Om})\geq \sigma$, we have that $\mu(\{x_s\})\geq \sigma$. This implies that $\mu(\mc A_{\sigma})= +\infty$, however
\[ \mu(\mc A_{\sigma})= \lim_{k\ra+\infty} \int_{\mc A_{\sigma}}|\na u_k|^{n} dx \leq C.\]
Thus $\mc A_{\sigma}=\{x_1, x_2,\cdots, x_p\}$.

\noi {\bf Assertion 1}. If we choose $\sigma>0$ such that
$\sigma^{\frac{1}{n-1}}<r_1$, then we have
\[ \lim_{k\ra \infty}\int_{K} f(x,u_k)u_k~ dx= \int_{K} f(x,u)u~ dx,\]
for any relative compact subset $K$ of $\overline{\Om} \setminus \mc A_{\sigma}$.\\
Indeed, let $x_0\in K$ and
$r_0>0$ be such that $\mu(\textbf{B}_{r_0}(x_0)\cap \oline{\Om})<\sigma$. Consider a function $\phi\in C^{\infty}_{0}(\Om,[0,1])$ such that $\phi\equiv1$ in $\textbf{B}_{\frac{r_0}{2}}(x_0)\cap \overline{\Om}$ and $\phi\equiv0$ in $\overline{\Om}\setminus (\textbf{B}_{r_0}(x_0)\cap\oline{\Om})$.
Thus
\[\lim_{k\ra\infty} \int_{\textbf{B}_{r_0}(x_0)\cap \overline{\Om}}|\na u_k|^n \phi ~dx = \int_{\textbf{B}_{r_0}(x_0)\cap \overline{\Om}}\phi d\mu \leq \mu(\textbf{B}_{r_0}(x_0)\cap \overline{\Om})<\sigma.\]
Therefore for $k\in \mb N$ sufficiently large  and $\e>0$ sufficiently small, we have
\[ \int_{\textbf{B}_{\frac{r_0}{2}}(x_0)\cap \overline{\Om}}|\na u_k|^n ~dx = \int_{\textbf{B}_{\frac{r_0}{2}}(x_0)\cap \overline{\Om}}|\na u_k|^n\phi dx \leq (1-\e)\sigma,\]
which together with implies {\small\begin{align}\label{ed9}
\int_{\textbf{B}_{\frac{r_0}{2}}(x_0)\cap \overline{\Om}}
|f(x,u_k)|^q =\int_{\textbf{B}_{\frac{r_0}{2}}(x_0)\cap
\overline{\Om}}|h(x,u_k)|^{q} e^{q|u_k|^{\frac{n}{n-1}}} \leq d
\int_{\textbf{B}_{\frac{r_0}{2}}(x_0)\cap \overline{\Om}}
e^{(1+\de)q|u_k|^{\frac{n}{n-1}}}\leq K
\end{align}}
if we choose $q>1$ sufficiently close to $1$ and $\de>0$ is small
enough such that $\frac{(1+\de)q \sigma^{\frac{1}{n-1}}}{r_1}<1$.
Now we estimate
\[\small{\int_{\textbf{B}_{\frac{r_0}{2}}(x_0)\cap \overline{\Om}}|f(x,u_k)u_k-f(x,u)u| ~dx \leq I_1+ I_2}\]
where \[\small {I_1 := \int_{\textbf{B}_{\frac{r_0}{2}}(x_0)\cap \overline{\Om}}|f(x,u_k)- f(x,u)||u| dx \;\mbox{and}\; I_2 := \int_{\textbf{B}_{\frac{r_0}{2}}(x_0)\cap \overline{\Om}}|f(x,u_k)||u_k-u|~ dx.}\]
\noi Note that, by H\"{o}lder's inequality and \eqref{ed9},
\[I_2 = \int_{\textbf{B}_{\frac{r_0}{2}}(x_0)\cap \overline{\Om}}|f(x,u_k)||u_k-u|dx \leq K \left(\int_{\Om}|u_k-u|^{q^{\prime}}\right)^{1/{q^{\prime}}}\ra 0\; \mbox{as}\; k\ra \infty.\]
Now, we claim that $I_1\ra 0$. Indeed, given $\e>0$, by density we can take $\phi\in C_{0}^{\infty}(\Om)$ such that $\|u-\phi\|_{q^{\prime}}<\e$. Thus,
{\small\begin{align*}
\int_{\textbf{B}_{\frac{r_0}{2}}(x_0)\cap \overline{\Om}}|f(x,u_k)- f(x,u)||u| & \leq \int_{\textbf{B}_{\frac{r_0}{2}}(x_0)\cap \overline{\Om}}(|f(x,u_k)||u-\phi| + |f(x,u_k)- f(x,u)||\phi|) \\
& \quad\quad +\int_{\textbf{B}_{\frac{r_0}{2}}(x_0)\cap \overline{\Om}}|f(x,u)||\phi-u|.
\end{align*}}
Applying H\"{o}lder inequality and using equation \eqref{ed9} , we have
\[\small{\int_{\textbf{B}_{\frac{r_0}{2}}(x_0)\cap \overline{\Om}}|f(x,u_k)||u-\phi| ~dx\leq \left(\int_{\textbf{B}_{\frac{r_0}{2}}(x_0)\cap \overline{\Om}}|f(x,u_k)|^q dx\right)^{1/q}\|u-\phi\|_{q^{\prime}}<\e.}\]
\noi Using Lemma \ref{lc1}, we have
\[\small{\int_{\textbf{B}_{\frac{r_0}{2}}(x_0)\cap \overline{\Om}}|f(x,u_k)- f(x,u)||\phi|~ dx \leq \|\phi\|_{\infty}\int_{\textbf{B}_{\frac{r_0}{2}}(x_0)\cap \overline{\Om}}|f(x,u_k)- f(x,u)| dx\ra 0.}\]
Also from equation \eqref{ed9}, we have
$\ds \int_{\textbf{B}_{\frac{r_0}{2}}(x_0)\cap \overline{\Om}}|f(x,u)||\phi- u| dx\ra 0,$
and hence the claim. Now to conclude Assertion 1 we use that $K$ is compact and we repeat the same procedure over a finite covering of balls.\\
\noi {\bf Assertion 2:} Let $\e_0>0$ be such that $B_{\e_0}(x_i)\cap B_{\e_0}(x_j)=\emptyset$ if $i\neq j$ and
$\Om_{\e_0}=\{x\in\bar\Om:|x-x_j|\geq\e_0, j=1,2,...,m\}$. Then
\begin{equation}\label{n7a5}
 \int_{\Om_{\e_0}}(|\nabla u_k|^{n-2}\nabla u_k-|\nabla u|^{n-2}\nabla u)(\nabla u_k-\nabla u)\ra0.
 \end{equation}
Indeed, let $0<\e<\e_0$ and $\phi\in C^\infty_c(\mb R^n)$ such that
$\phi\equiv1$ in $B_{1/2}(0)$ and $\phi\equiv0$ in $\bar\Om\setminus B_1(0)$.
Take $\psi_\e=1-\ds\sum_{j=1}^m\phi\left(\frac{x-x_j}{\e}\right)$. Then $0\leq\psi_\e\leq 1$, $\psi_\e\equiv 1$ in
 $\bar\Om_\e=\bar\Om\setminus\cup_{j=1}^m B_\e(x_j)$, $\psi_\e\equiv0$ in $\cup_{j=1}^m B_{\e/2}(x_j)$
and $\{\psi_\e u_k\}$ is bounded in $W^{1,n}_{0}(\Om)$. Now taking $v=\psi_\e u_k$ in (\ref{n7a2}) we get
\begin{equation}\label{n7a6}
  m(\|u_k\|^n)\int_\Om\left[|\nabla u_k|^n\psi_\e+|\nabla u_k|^{n-2}\nabla u_k\nabla \psi_\e u_k\right]- \int_\Om f(x,u_k)u_k\psi_\e
\leq \e_k\|\psi_\e u_k\|.
\end{equation}
Again taking $v=-\psi_\e u$ in (\ref{n7a2}) we get
\begin{align}\label{n7a7}
  m(\|u_k\|^n)\int_\Om\left[-|\nabla u_k|^{n-2}\nabla u_k\nabla u\psi_\e-|\nabla u_k|^{n-2}\right.&\left.\nabla u_k\nabla \psi_\e u\right]
+\int_\Om f(x,u_k)u\psi_\e\notag\\
&\leq \e_k\|\psi_\e u\|.
\end{align}
Also using the convexity of $t\mapsto |t|^n$ for $t\in \mb R^n$ and  $m(t)\ge m_0>0$, we have
\begin{equation}\label{n7a8}
  0\leq m(\|u_k\|^n) \int_{\Om}\left(|\nabla u_k|^{n}-|\nabla u_k|^{n-2}\nabla u_k\nabla u-|\nabla u|^{n-2}\nabla u\nabla u_k
+|\nabla u|^n \right)\psi_\e,
\end{equation}
from (\ref{n7a6}), (\ref{n7a7}) and (\ref{n7a8}) we get
{\small\begin{align*}
 0\leq m(\|u_k\|^n)\int_{\Om}&|\nabla u_k|^{n-2}\nabla u_k\nabla\psi_\e(u_k-u)+ 2\e_k\|\psi_\e u_k\|\\
&+m(\|u_k\|^n)\int_{\Om}\psi_\e|\nabla u|^{n-2}\nabla u(\nabla u-\nabla u_k)+ \int_{\Om}f(x,u_k)(u_k-u)\psi_\e.
\end{align*}}
Now as by Young's inequality, for given $\de>0$, there exists $C_\de>0$ such that
\begin{align}\label{n7a9}
m(\|u_k\|^n) \int_{\Om}|\nabla u_k|^{n-2}&\nabla u_k\nabla\psi_\e(u_k-u) \leq\delta\int_\Om|\nabla u_k|^n
 +C_\de\int_\Om|\nabla\psi_\e|^n|u_k-u|^n\nonumber\\
&\leq \de C+C_\de(\int_\Om|\nabla \psi_\e|^{nr})^{1/r}(\int_\Om|u_k-u|^{ns})^{1/s}
\end{align}
where $C$, $r$ and $s$ are positive real number such that $\ds\frac{1}{r}+\ds\frac{1}{s}=1$.
Thus using this and boundedness of $\{u_k\}$, we get
\begin{equation}\label{n7a10}
 \ds\limsup_{k\ra\infty} m(\|u_k\|^n)\int_{\Om}|\nabla u_k|^{n-2}\nabla u_k\nabla\psi_\e(u_k-u)\leq 0.
\end{equation}
Also noting that $u_k\rightharpoonup u$ weakly in $W^{1,n}_{0}(\Om)$ and $m(\|u_k\|^n)$ bounded,
we have
\begin{equation}\label{1n7a11}
  \lim_{k\ra\infty}m(\|u_k\|^n)\int_{\Om}\psi_\e|\nabla u|^{n-2}\nabla u(\nabla u-\nabla u_k)=0.
\end{equation}
By Assertion 1, taking $\mathcal{K}=\oline{\Om}_{\e/2}$ one can check that
\begin{equation}\label{n7a11}
 \lim_{k\ra\infty}\int_{\Om}f(u_k)(u_k-u)\psi_\e=0.
\end{equation}
Now from  (\ref{n7a9})-(\ref{n7a11}), (\ref{n7a5}) follows.
Since $\e_0$ is arbitrary, we get $\nabla u_k(x)\ra\nabla u(x)$ a.e in
$\Om$ and hence $|\nabla u_k|^{n-2}\nabla u_k\rightharpoonup|\nabla u|^{n-2}\nabla u$ weakly in $(L^{n/n-1}(\Om))^n$.\QED

\noi Now we define the Nehari manifold associated to the functional $J$, as
\[ \mc{N}:=\{0\not\equiv u\in W_0^{1,n}(\Omega):\langle J^{\prime}(u),u \rangle=0\}\]
and let $\ds b:=\inf_{u\in \mc{N}} J(u)$. Then we need the following to compare  $c_*$ and $b$.
\begin{Lemma}\label{lem7.3}
If condition $(f1)$ holds, then for each $x\in\Omega$, $ sf(x,s)-2nF(x,s)$ is increasing for $s\ge0$. In particular $sf(x,s)-2nF(x,s)\geq 0$ for all $(x,s)\in \Omega\times [0,\infty)$.
\end{Lemma}
\proof Suppose $0<s<t$. Then for each $x\in\Om$, we obtain
 \begin{align*}
sf(x,s)-2nF(x,s)&=\frac{f(x,s)}{s^{2n-1}}s^{2n}-2nF(x,t)+2n\int_s^t f(x,\tau)d\tau\\
&< \frac{f(x,t)}{t^{2n-1}}s^{2n} -2n F(x,t)+\frac{f(x,t)}{t^{2n-1}}(t^{2n}-s^{2n})\\
&\leq tf(x,t)-2nF(x,t),
\end{align*}
which completes the proof.\QED
\begin{Lemma}:
If (i) $\frac{m(t)}{t}$ is nonincreasing for $t> 0$ (ii) for each $x\in\Omega, \frac{f(x,t)}{t^{2n-1}}$ is increasing for $t>0$ hold.
Then $c_*\leq b $.
\end{Lemma}
\proof Let $u\in \mc{N}$, define $h:(0,+\infty)\rightarrow \mb R$ by $ h(t)=J(tu)$. Then
\[h^{\prime}(t)=\langle J^{\prime}(tu),u\rangle=m(t^n\|u\|^n)t^{n-1}\|u\|^n-\int_\Om f(x,tu)u~dx \; \text{for all} \; t>0.\]
Since $\langle J^{\prime}(u),u\rangle=0$, we have {\small
\begin{align*} h^{\prime}(t)=&\|u\|^{n} t^{2n-1}\left(\frac{m(t^n\|u\|^n)}{t^n\|u\|^n}-\frac{m(\|u\|^n)}{\|u\|^n}\right)\\
&\quad +t^{2n-1}\int_{\Om}\left(\frac{f(x,u)}{u^{2n-1}}-\frac{f(x,tu)}{(tu)^{2n-1}}\right)u^{2n}dx.\end{align*} }
So $h^{\prime}(1)=0$, $h^{\prime}(t)\geq 0$ for $0<t<1$ and $h^{\prime}(t)<0$ for $t>1$. Hence $J(u)=\ds \max_{t\geq0}J(tu)$. Now define $g:[0,1]\rightarrow W_0^{1,n}(\Omega)$ as  $g(t)=(t_0 u)t$, where $t_0$ is such that $J(t_0 u)<0$. We have $g\in \Gamma$ and therefore
\[ c_*\leq\max_{t\in[0,1]}J(g(t))\leq \max_{t\geq 0}J(tu)=J(u).\]
  Since $u\in \mc N$ is arbitrary, $c_*\leq b$ and the proof is complete.\QED

\noi We recall the following result of  Lions \cite{li} known as higher integrability Lemma.
\begin{Lemma}\label{plc}
Let $\{v_k : \|v_k\|=1\}$ be a sequence in $W^{1,n}_{0}(\Om)$
converging weakly to a non-zero $v$. Then for every $p$ such that
$1<p<(1-\|v\|^n)^{\frac{-1}{n-1}}$,\[\sup_{k}\int_{\Om} e^{p \al_{n}
|v_k|^{\frac{n}{n-1}}}< \infty.\]
\end{Lemma}
\noi {\bf Proof of Theorem \ref{thm711}:} Let $\{u_k\}$ be a Palais-Smale sequence at level $c_*$. That is $J(u_k)\rightarrow c_*$ and $J^\prime (u_k) \rightarrow 0$. Then by Lemma 2.4 and Lemma 2.7, there exists $u_0\in W^{1,n}_{0}(\Om)$ such that $u_k \rightharpoonup u_0$ weakly in $W^{1,n}_{0}(\Om)$, $\na u_k (x) \rightarrow \na u_0(x)$ a.e. in $\Om$. Now we claim that $u_0$ is the required positive solution.\\
\noi \textbf{claim 1:} $u_0>0$ in $\Omega$.\\
 \proof As $\{u_k\}$ is bounded, so up to a subsequence $\|u_k\|\rightarrow \rho_0>0$. Moreover, condition $J^{\prime}(u_k)\rightarrow 0$ and Lemma \ref{lem714} implies that
 \begin{equation}\label{n7a12}
 m(\rho_0^n)\int_\Omega |\nabla u_0|^{n-2}\nabla u_0\nabla v~ dx=\int_\Om f(x,u_0)v ~dx\;  \text{for all} \; v\in W_0^{1,n}(\Omega).
 \end{equation}
  That is $u_0$ satisfies $ -\De_n u_0=\frac{1}{m(\rho_0^n)}f(x,u_0)\;\mbox{in}\;\Om,\quad u=0\;\mbox{on}\;\partial\Om.$
 Using the growth condition of $f(x,t)$ and Trudinger-Moser inequality, we get $f(.,u_0)\in L^p(\Omega)$ for all $1\leq p\le \infty$. Therefore by regularity theory $u_0\in C^{1,\al}(\overline{\Omega})$ and hence by strong maximum principle, we get $u_0>0$ in $\Omega$ and hence the claim.

\noi \textbf{claim 2:} $\ds m(\|u_0\|^n)\|u_0\|^n\geq \int_\Om f(x,u_0)u_0 dx$.\\
\proof Suppose by contradiction that $m(\|u_0\|^n)\|u_0\|^n< \int_\Om f(x,u_0)u_0 ~dx$. That is, $\langle J^{\prime}(u_0),u_0\rangle<0.$
 Using \eqref{n1} and Sobolev imbedding, we can see that $\langle J^{\prime}(tu_0),u_0\rangle>0$ for t sufficiently small. Thus there exist $\sigma\in(0,1)$ such that $\langle I^{\prime}(\sigma u_0),u_0\rangle=0$. That is,  $\sigma u_0\in \mc N$. Thus according to  Lemma \ref{lem7.3},
 \begin{align*}
 c_*&\leq b \leq J(\sigma u_0)=J(\sigma u_0)-\frac{1}{2n}\langle J^{\prime}(\sigma u_0),\sigma u_0\rangle\\
 &=\frac{M(\|\sigma u_0\|^n)}{n}-\frac{m(\|\sigma u_0\|^n)\|\sigma u_0\|^n}{2n}+\int_\Om \frac{(f(x,\sigma u_0)\sigma u_0-2nF(x,
 \sigma u_0))}{2n}\\
 &<\frac{1}{n}M(\| u_0\|^n)-\frac{1}{2n}m(\|u_0\|^n)\| u_0\|^n+\frac{1}{2n}\int_\Om (f(x, u_0)u_0-2nF(x,u_0))
 \end{align*}
 By lower semicontinuity of norm and Fatou's Lemma, we get
 \begin{align*}
 c_*&<  \liminf_{k\rightarrow\infty}\frac{1}{n}\left(M(\|u_k\|^n)-\frac{1}{2}m(\|u_k\|^{n})\| u_k\|^{n}\right)\\
 &\quad+\liminf_{k\rightarrow\infty}\frac{1}{2n}\int_\Om [f(x, u_k)u_k-2nF(x,u_k)]dx\\
 &\leq \lim_{k\rightarrow\infty}[J(u_k)-\frac{1}{2n}\langle J^{\prime}(u_k),u_k\rangle]=c_*,
 \end{align*}
 which is a contradiction and the claim 2 is proved.\\
\noi \textbf{Claim 3:} $J(u_0)=c_*$.\\
\proof Using $ \int_\Om F(x,u_k)\rightarrow \int_\Om F(x,u_0)$ and lower semicontinuity of norm we have $J(u_0)\leq c_*$.
 We are going to show that the case $J(u_0)<c_*$ can not occur.\\
 Indeed,  if $J(u_0)<c_*$ then $\|u_0\|^n<\rho_0^n.$ Moreover,
 \begin{equation}\label{7a3}
   \frac{1}{n}M(\rho_0^n)=\lim_{k\rightarrow\infty}\frac{1}{n}M(\|u_k\|^n)=c_*+\int_\Om F(x,u_0)dx,
 \end{equation}
  which implies  $\ds \rho_0^n = M^{-1}(nc_*+n\int_\Om F(x,u_0)dx)$.
  Next defining $v_k=\frac{u_k}{\|u_k\|}$ and $v_0=\frac{u_0}{\rho_0}$, we have $v_k\rightharpoonup v_0$ in $W_0^{1,n}(\Omega)$ and $\|v_0\|<1$. Thus by Lion's lemma \ref{plc},
\begin{equation}\label{7a4}
 \sup_{k\in \mb N}\int_\Omega e^{p|v_k|^{\frac{n}{n-1}}}~dx<\infty\;\;  \text{for all}\;  1<p<\frac{\alpha_n}{(1-\|v_0\|^n)^\frac{1}{n-1}}.
 \end{equation}
\noi On the other hand, by Assertion 2, \eqref{7a0} and Lemma \ref{lem7.3}, we have
\[  J(u_0)\ge \frac{M(\|u_0\|^2)}{n}-\frac{m(\|u_0\|^2)\|u_0\|^2}{2n} +\int_\Om \frac{(f(x,u_0)u_0-2n F(x,u_0)}{2n}.\]
So, $J(u_0)\geq 0$. Using this together with Lemma \ref{lem7.2} and the equality,
$n(c_*-J(u_0))=M(\rho_{0}^{n})-M(\|u_0\|^n)$
we get
$M(\rho_{0}^n) \le n c_*+M(\|u_0\|^n)<M(\al_{n}^{n-1})+M(\|u_0\|^n)$
and therefore by $(m1)$
\begin{equation}\label{7a6}
\rho_{0}^{n}< M^{-1}\left(M(\al_{n-1}^{n})+M(\|u_0\|^n)\right)\le \al_{n}^{n-1}+\|u_0\|^n.
\end{equation}
Since $\rho_{0}^{n}(1-\|v_0\|^n)=\rho_{0}^{n}-\|u_0\|^n$, from \eqref{7a6} it follows that
\[ \rho_{0}^{n}< \frac{\al_{n}^{n-1}}{1-\|v_0\|^n}.\]
Thus, there exists $\ba>0$ such that $ \|u_k\|^{\frac{n}{n-1}} < \ba < \frac{\al_{n}}{(1-\|v_0\|^n)^{\frac{1}{n-1}}}$ for $k$ large. We can choose $q>1$ close to $1$ such that $q \|u_k\|^{\frac{n}{n-1}} \le  \ba < \frac{\al_{n}}{(1-\|v_0\|^n)^{\frac{1}{n-1}}}$ and using \eqref{7a4}, we conclude that for $k$ large
\[ \int_\Om e^{q |u_k|^{n/n-1}} dx \le \int_\Om e^{\ba |v_k|^{n/n-1}} \le C.\]
Now by standard calculations, using H\"{o}lder's inequality and weak convergence of  $\{u_k\}$ to $u_0$, we get $\int_\Om f(x,u_k) (u_k-u_0) \rightarrow 0$ as $k\rightarrow \infty$. Since $\langle J^\prime (u_k), u_k-u_0\rangle \rightarrow 0$, it follows that
\begin{equation} \label{na7new}
m(\|u_k\|^n) \int_{\Om} |\na u_k|^{n-2} \na u_k (\na u_k-\na u_0) \rightarrow 0.
\end{equation}
\noi On the other hand, using $u_k\rightharpoonup u_0$ weakly and boundedness of $m(\|u_k\|^n)$,
\begin{equation}\label{na8new}
m(\|u_k\|^n) \int_{\Om} |\na u_0|^{n-2} \na u_0 (\na u_k - \na u_0)\ra 0\; \mbox{as}\; k\ra\infty.
\end{equation}
Subtracting \eqref{na8new} from \eqref{na7new}, we get
\[m(\|u_k\|^{n}) \int_\Om (|\na u_k|^{n-2}\na u_k - |\na u_0|^{n-2} \na u_0)\cdot (\na u_k -\na u_0) \rightarrow 0 \]
as $k\rightarrow \infty$. Now using this and the following
inequality
\begin{align}\label{11e2}
|a-b|^l\leq 2^{l-2}(|a|^{l-2}a-|b|^{l-2}b)(a-b)\;\mbox{for all}\; a,b\in \mb R^n\;\mbox{and}\; l\geq 2,
\end{align}
 with $a=\na u_k$ and $b=\na u_0$, we obtain
\[m(\|u_k\|^n)\int_{\Om}|\na u_k- \na u_0|^n \ra 0\;\mbox{as}\; k\ra\infty.\]
Since $m(t)\ge m_0$, we obtain $u_k\ra u$ strongly in $W^{1,n}_{0}(\Om)$ and hence $\|u_k\| \rightarrow \|u_0\|$.
Therefore, $J(u_0)=c_*$ and hence the claim.

\noi Now By Assertion 3 and \eqref{7a3} we can see that $M(\rho_{0}^{n})=M(\|u_0\|^n)$ which shows that $\rho_{0}^{n}=\|u_0\|^n$. Hence by \eqref{n7a12} we have
\[  m(\|u_0\|^n) \int_\Om |\na u_0|^{n-2}\na u_0 \na v~ dx =\int_\Om f(x,u_0) v ~dx, \; \text{for all} \; v\in W^{1,n}_{0}(\Om).\]
Thus, $u_0$ is a solution of $(\mc M)$. \QED

\section{Convex-Concave type nonlinearities}
\setcounter{equation}{0}
In this section, we study the existence and multiplicity of  solutions for the following problem
$$ (P_{\la, M})\quad \left\{
\begin{array}{lr}
 \quad  -m(\int_{\Om}|\na u|^n)\De_n u = \la h(x)|u|^{q-1}u+ u|u|^{p} ~ e^{|u|^{\ba}} \; \text{in}\;
\Om \\
\quad \quad \quad \quad u \geq 0 \; \mbox{in}\; \Om,\quad u\in W^{1,n}_{0}(\Om),\\
 \quad \quad\quad \quad\quad u =0\quad\quad \text{on} \quad \partial \Om
 \end{array}
\right.
$$
 where  $0< q<n-1< 2n-1< p+1$, $\ba\in (1,\frac{n}{n-1}]$ and $\la>0$.
 Let $\gamma=\frac{n}{n-q-1}$, $k=\frac{p+2+\ba}{q+1}>1$ and $k^\prime=\frac{k}{k-1}$. We assume the following:
 \begin{enumerate}
 \item[$(A1)$] $m(s)=as+b$, where $a,$ $b>0$.
\item[$(A2)$] $\ds h\in L^{\ga}(\Om)$, $h^{+}\not\equiv 0$, $h$ can be indefinite and vanish in some open subset of $\Om$.
\end{enumerate}
\noi We show the following existence and multiplicity result in the subcritical case:
\begin{Theorem}\label{zth11}
Let $\ba\in \left(1,\frac{n}{n-1}\right)$. Then there exists $\la_0>0$ such that
for $\la\in(0,\la_0)$, $(P_{\la,M})$ admits at least two
solutions.
\end{Theorem}
\noi In the critical case, we show the following existence result:
 \begin{Theorem} \label{zth12}
Let $\ba=\frac{n}{n-1}$, then there exist $\la_{00}>0$ such that for $\la\in (0,\la_{00})$, $(P_{\la,M})$ admits a solution.
\end{Theorem}
\subsection{The Nehari manifold and fibering maps }
\noi The Euler functional associated with the problem $(P_{\la, M })$ is $J_{\la,M} :W^{1,n}_{0}(\Om) \lra \mb{R}$ defined as
\begin{equation}\label{zeq1}
J_{\la,M}(u)= \frac{1}{n} M(\|u\|^n)- \frac{\la}{q+1} \int_{\Om}h(x)|u|^{q+1} dx - \int_{\Om}
G(u) dx,
\end{equation}
where $g(u)=u|u|^{p}e^{|u|^{\ba}}$, $\ds G(u)=\int_{0}^{u} g(s)
ds$ and $\ds M(u)=\int_{0}^{u} m(s) ds$, .
\begin{Definition}
We say that $u\in W^{1,n}_{0}(\Om)$ is a weak solution of
$(P_{\la,M})$ if for all $\phi \in W^{1,n}_{0}(\Om)$, we have
\begin{align}\label{zdef1}
\ds m(\|u\|^n)\int_{\Om}|\na u|^{n-2}\na u \na \phi dx = \int_{\Om}g(u)\phi dx +\la \int_{\Om}
h(x)|u|^{q-1}u\phi dx.
\end{align}
\end{Definition}
\noi For $u\in W^{1,n}_{0}(\Om)$, we define the fiber map $\phi_{u, M}: \mb
R^+\ra \mb R$ as
\begin{align*}
\phi_{u,M}(t) &= J_{\la,M}(tu) = \frac{t^n}{n} M(\|u\|^n)- \frac{\la
 t^{q+1}}{q+1} \int_{\Om}h(x)|u|^{q+1}dx -\int_{\Om} {G(tu)} dx.
 \end{align*}
 Also
 \begin{align*}
\phi_{u,M}^{\prime}(t) &= t^{n-1} m(\|u\|^n)- {\la
 t^{q}} \int_{\Om}h(x)|u|^{q+1}dx - \int_{\Om}{g(tu)
u} dx,\\
\phi_{u,M}^{\prime\prime}(t) &= (n-1) t^{n-2} m({\|tu\|^n})\|u\|^n+ nt^{2n-2} m^{\prime}(\|tu\|^{n})\|u\|^{2n}\\
&\quad - q \la t^{q-1} \int_{\Om}h(x)|u|^{q+1}dx  - \int_{\Om}{g^{\prime}(tu)u^2}.
\end{align*}

\noi It is easy to see that the energy functional $J_{\la,M}$ is not
bounded below on the space $W^{1, n}_{0}(\Om)$. But we will show that it is bounded below
on an appropriate subset of $W^{1, n}_{0}(\Om)$ and a minimizer on
subsets of this set gives rise to solutions of $(P_{\la, M})$. In
order to obtain the existence results, we define the Nehari
manifold
\begin{equation*}
\mc N_{\la,M}:= \left\{u\in W^{1, n}_{0}(\Om): \ld
J_{\la,M}^{\prime}(u),u\rd=0 \right\}=\left\{u\in W^{1, n}_{0}(\Om) :
\phi_{u,M}^{\prime}(1)=0\right\}
\end{equation*}
where $\ld\;,\; \rd$ denotes the duality between $W^{1,n}_{0}(\Om)$
and its dual space. Therefore $u\in \mathcal N_{\la,M}$ if and only if
\begin{equation}\label{zeq2}
m(\|u\|^n)- \la \int_{\Om}h(x)|u|^{q+1} dx - \int_{\Om}{g(u) u}dx =0
.
\end{equation}
We note that $\mathcal N_{\la,M}$ contains every solution of
$(P_{\la, M})$. One can easily see that $tu\in \mathcal N_{\la,M}$ if and only if
$\phi_{u,M}^{\prime}(t)=0$ and in particular, $u\in \mc N_{\la,M}$ if
and only if $\phi_{u,M}^{\prime}(1)=0$. Also
\begin{align*}
\mathcal N_{\la,M}^{\pm}&:= \left\{u\in \mc N_{\la,M}:
\phi_{u,M}^{\prime\prime}(1)
\gtrless0\right\} =\left\{tu\in W^{1, n}_{0}(\Om) : \phi_{u,M}^{\prime}(t)=0,\; \phi_{u,M}^{\prime\prime}(t)\gtrless  0\right\},\\
\mathcal N_{\la,M}^{0}&:= \left\{u\in \mc N_{\la,M}:
\phi_{u,M}^{\prime\prime}(1) = 0\right\}=\left\{tu\in W^{1,n}_{0}(\Om) : \phi_{u,M}^{\prime}(t)=0,\; \phi_{u,M}^{\prime\prime}(t)= 0\right\}.
\end{align*}
Let $H(u)=\int_{\Om} h|u|^{q+1} dx$. Then we define
$H^{\pm}:=\{u\in W^{1,n}_{0}(\Om): H(u)\gtrless 0\}$, $H_{0}:=\{u\in W^{1,n}_{0}(\Om): H(u)= 0\}$,
 and $H^{\pm}_{0}: =H^{\pm}\cup H_{0}$.\\

\noi Now we describe the behavior of the fibering map $\phi_{u,M}$
according to the sign of $H(u)$.

\noi Case 1: $u\in H^{-}_{0}$.\\
In this case, firstly we define $\psi_{u}: \mb R^{+} \lra \mb R$ by
\begin{equation}\label{z1}
\psi_{u}(t)= t^{n-1-q} m(\|tu\|^n)-  t^{-q}\int_{\Om}g(tu)u dx.
\end{equation}
Clearly, for $t>0$, $tu\in \mc N_{\la,M}$ if and only if $t$ is a
solution of
 $\psi_{u}(t)={\la} \int_{\Om}h(x)|u|^{q+1}.$
 \begin{align}
\psi_{u}^{\prime}&(t) = (n- 1-q) t^{(n-2-q)} m(\|tu\|^n)\|u\|^n + n t^{2n-2-q} m^{\prime}(\|tu\|^n)\|u\|^{2n} - t^{-q}
\int_{\Om}g^{\prime}(tu)u^2 \label{zeq03}\\
=&(2n- 1-q) t^{2n-2-q} a \|u\|^{2n} + (n-1-q) b t^{n-2-q}\|u\|^n - (1+p-q)t^{-1-q}\int_{\Om}g(tu)u\notag\\
&\quad\quad\quad - \ba t^{-q-1+\ba}\int_{\Om}|u|^{\ba}g(tu)u.\label{zeq04}
\end{align}
Therefore $\psi_{u}^{\prime}(t)<0$ for all $t>0$.
As $u\in H^{-}_{0}$ so there exists $t_*(u)$ such that $\psi_{u}(t_*)={\la}
\int_{\Om}h(x)|u|^{q+1}.$ Thus for $0<t<t_*$, $\phi_{u,M}^{^{\prime}}(t)= t^q(\psi_{u}(t)-{\la}
\int_{\Om}h(x)|u|^{q+1})>0$ and for $t>t_*$, $\phi_{u,M}^{^{\prime}}(t)<0.$ Hence $\phi_{u,M}$ is
increasing on $(0,t_*)$, decreasing on $(t_*, \infty)$. Since
$\phi_{u,M}(t)>0$ for $t$ close to $0$ and $\phi_{u,M}(t)\ra -\infty$ as
$t\ra \infty$, we get $\phi_{u,M}$ has exactly one critical point
$t_{1}(u)$, which is a global maximum point. Hence $t_{1}(u)u \in
\mc N^{-}_{\la,M}$.

\noi Case 2: $u\in H^+$.

\noi In this case, we claim that there exists $\la_0>0$ and a unique
$t_*$ such that for $\la\in(0,\la_0)$, $\phi_u$ has exactly two
critical points $t_1(u)$ and $t_2(u)$ such that
$t_1(u)<t_*(u)<t_2(u)$, and moreover $t_1(u)$ is a local minimum
point and $t_2(u)$ is a local maximum point. Thus $t_1(u)u\in\mc
N_{\la,M}^{+}$ and $t_2(u)u\in\mc N_{\la,M}^{-}$.

\noi To show this we need following Lemmas:
\begin{Lemma}\label{zle3}
Let $\ds \La: = \left\{u\in W^{1,n}_{0}(\Om) \;|\;
\|u\|^{\frac{3n}{2}} \leq \frac{\int_{\Om}{g^{\prime}(u){u}^2}dx}{2\sqrt{ab(2n-1-q)(n-1-q)}}\right\}$.
Then there exists $\la_0>0$ such that
for every $\la\in(0,\la_0)$,
\begin{equation}\label{zeq8}
\La_m:=\inf_{u\in \La\setminus\{0\} \cap H^{+}_{0}}\left\{\int_{\Om}
\left(p+2-2n +\ba{|u|^{\ba}}\right)|u|^{p+2} e^{|u|^{\ba}}- (2n-1-q)
{\la} \int_{\Om}h(x)|u|^{q+1} \right\}>0.
\end{equation}
\end{Lemma}
\proof
 Step 1: $\ds \inf_{u\in \La\setminus\{0\}\cap H^{+}_{0}} \|u\|
>0$. Suppose this is not true. Then we find a sequence $\{u_k\}
\subset \La\setminus\{0\}\cap H^{+}_{0}$ such that $\|u_k\| \ra 0$
and we have
\begin{align}\label{zeq9}
 \|u_k\|^{\frac{3n}{2}}\leq  \left(\frac{1}{2\sqrt{ab(2n-1-q)(n-1-q)}}\right)
\int_{\Om}{g^{\prime}(u_k)u_{k}^2}\;dx\;\;\; \fa \;\;k.
\end{align}
From $g(u)=u|u|^pe^{|u|^{\ba}}$, H\"{o}lder's inequality and Sobolev
inequality, we have
\begin{align*}
\int_{\Om} {g^{\prime}(u_k)u_{k}^2} dx &=
\int_{\Om} {\left(p+1+\ba|u_k|^{\ba}\right)|{u_k}|^{p+2}e^{|u_k|^{\ba}}} dx\\
&\leq C\int_{\Om}|{u_k}|^{p+2} e^{(1+\de)|u_k|^{\ba}} dx\\
 &\leq C\left(\int_{\Om}{|{u_k}|^{(p+2)t^{'}}} dx
\right)^{\frac{1}{t^{'}}}\left( \int_{\Om} e^{t(1+\de)|u_k|^{\ba}}  dx\right)^{\frac{1}{t}}\\
&\leq C' \|u_k\|^{p+2} \left(\sup_{\|w_k\|\leq 1}
\int_{\Om}e^{t(1+\de)\|u_k\|^{\ba}|w_k|^{\ba}}
dx\right)^{\frac{1}{t}},
\end{align*}
since  $\|u_k\|\ra 0$ as $k\ra\infty$, we can choose $\ds\al=
t(1+\delta){\|u_k\|}^{\ba}$ such that $\alpha \leq {\alpha_n}$.
Hence by this, \eqref{zeq9}, we obtain $1\leq
K'\|u_k\|^{p+2-\frac{3n}{2}}\ra 0$ as $k\ra \infty$, since $p+2 > \frac{3n}{2}$,
which gives a contradiction.\\
\noi Step 2: Let $\ds C_1=\inf_{u\in \La\setminus\{0\} \cap
H^{+}_{0}}\ds\int_{\Om} {\left(p+2-2n+\ba{|u|^{\ba}}\right)|u|^{p+2}
e^{|u|^{\ba}}} dx$. Then $C_{1}>0$.

\noi From Step 1 and the definition of $\La$, we obtain
\begin{align*}
0&< \inf_{u\in \La\setminus\{0\} \cap H^{+}_{0}}\int_{\Om}
{g^{\prime}(u)u^2}dx =\inf_{u\in \La\setminus\{0\}\cap
H^{+}_{0}}\int_{\Om}{\left(p+1+\ba{|u|^{\ba}}\right)|u|^{p+2}
e^{|u|^{\ba}}} dx.
\end{align*}
Using this it is easy to check that
\begin{align*}
 \inf_{u\in \La \setminus\{0\}\cap H^{+}_{0}}\int_{\Om}
{\left(p+2-2n+\ba{|u|^{\ba}}\right)|u|^{p+2} e^{|u|^{\ba}}} dx>0.
\end{align*}
This completes step 2.

\noi Step 3: Let
$\la<\frac{1}{(2n-q-1)}(\frac{C_1}{l})^{\frac{(k-1)}{k}}$, where
$l=\int_{\Om}|h(x)|^{\frac{k}{k-1}}dx$. Then (\ref{zeq8}) holds.

\noi Using H\"{o}lder's inequality and $(A2)$ we have,
\begin{align*}
\int_{\Om}h(x)|u|^{q+1}
&\leq  \left(\int_{\Om}|h(x)|^{\frac{k}{k-1}}dx\right)^{\frac{k-1}{k}} \left(\int_{\Om} |u|^{(q+1)k} dx \right)^{\frac{1}{k}}\\
&= l^{\frac{k-1}{k}} \left(\int_{\Om} |u|^{p+2+\ba} dx \right)^{\frac{1}{k}}\\
&\leq l^{\frac{k-1}{k}}
\left(\int_{\Om}\left(p+2-2n+\ba{|u|^{\ba}}\right)
|u|^{p+2}e^{|u|^{\ba}} dx \right)^{\frac{1}{k}}\\
&\leq\left(\frac{l}{C_1}\right)^{\frac{k-1}{k}}
\int_{\Om}\left(p+2-2n+\ba{|u|^{\ba}}\right) |u|^{p+2}e^{|u|^{\ba}}
dx .
\end{align*}
The above inequality combined with step 2 proves the Lemma.\QED

The following Lemma completes the proof of claim made in case 2 above:

\begin{Lemma}\label{zle5}
Let $\la$ be such that
(\ref{zeq8}) holds. Then for every $u\in H^{+}\setminus\{0\}$, there
is a unique $t_*=t_{*}(u)>0$ and unique $t_{1}=
t_{1}(u)<t_{*}<t_{2}= t_{2}(u)$ such that $t_{1}u\in \mc
N_{\la,M}^{+}$, $t_{2}u\in\mc N_{\la,M}^{-}$  and $\ds J_{\la,M}(t_{1}u)=
\min_{0\leq t\leq t_{2}}J_{\la,M}(tu)$, $\ds J_{\la,M}(t_{2}u)=
\max_{t\geq t_{*}}J_{\la,M}(tu)$.
\end{Lemma}

\begin{proof}
 Fix $0\ne u\in H^{+}$. Then from \eqref{z1}, we note that $\psi_{u}(t)\ra -\infty$ as $t \ra\infty $,
from \eqref{zeq03} it is easy to see that $\ds \lim_{t\ra 0^{+}}\psi_{u}^{\prime}(t)>0$
and sum of second and third term in \eqref{zeq03} is a monotone
function in $t$. So there exists a unique $t_*=t_{*}(u)>0$ such that
$\psi_u(t)$ is increasing on $(0,t_*)$, decreasing on $(t_*, \infty)$
and $\psi_{u}^{\prime}(t_*)=0$. Using this and \eqref{zeq03}, we get
$t_*{u}\in \La\setminus\{0\} \cap H^{+}$. From
$t_{*}^{q+2}\psi_{u}^{\prime}(t_*)= 0$ and by definition of $\psi_u$, we
get

\begin{align*}
\psi_{u}(t_*)=\frac{1}{t_{*}^{q+1}(2n-1-q)}\left[ \int_{\Om}{g^{\prime}(t_{*}u)
(t_*u)^2}  dx -(2n-1) \int_{\Om}{g(t_*u) t_*u}  dx \right].
\end{align*}
Using Lemma \ref{zle3} and noting that
$g^\prime(s)s^2-(2n-1)g(s)s=(p+2-2n+\ba|s|^{\ba})|s|^{p+2}
e^{|s|^{\ba}}$, we have
\begin{align*}
\psi_{u}(t_*)- {\la} \int_{\Om}h(x)|u|^{q+1} &=\frac{1}{ t_{*}^{q+1} (2n-1-q)}\left[
\int_{\Om} \left( g^{\prime}(t_{*}u)
(t_*u)^2 -(2n-1)g(t_*u) t_*u  \right)dx  \right.\\
&\hspace{2cm} \left.-(2n-1-q) {\la} \int_{\Om}h |t_*u|^{q+1} \right]\\
&> \frac{\La_m}{t_{*}^{q+1} (2n-1-q)} >0.
\end{align*}

\noi Since $\psi_{u}(0)=0$, $\psi_{u}$ is increasing in $(0,t_*)$
and strictly decreasing in $(t_*,\infty)$, $\ds \lim_{t\ra
\infty}\psi_{u}(t)=-\infty$ and $u\in H^+$. Then there exists a
unique $t_{1}=t_{1}(u)<t_*$ and $t_{2}=t_{2}(u)>t_*$  such that
$\psi_{u}(t_{1})= \la \int_{\Om}h(x)|u|^{q+1} = \psi_{u}(t_{2})$
implies $t_{1}{u}$, $t_{2}{u}\in \mc N_{\la,M}$. Also
$\psi_{u}^{\prime}(t_{1})>0$ and $\psi_{u}^{\prime}(t_{2})<0$ give
$t_{1}u\in \mc N_{\la,M}^{+}$ and $t_{2}u\in \mc N_{\la,M}^{-}$.
Since $\phi_{u,M}^{\prime}(t)=t^q(\psi_{u}(t)- \la
\int_{\Om}h(x)|u|^{q+1})$. Then $\phi_{u,M}^{\prime}(t)<0$ for all
$t\in [0,t_1)$ and $\phi_{u,M}^{^{\prime}}(t)>0$ for all $t\in
(t_1,t_2)$ so $\ds\phi_{u,M}(t_1)=\min_{0\leq t\leq t_2}
\phi_{u,M}(t)$. Also $\phi_{u,M}^{\prime}(t)>0$ for all $t\in
[t_*,t_2)$, $\phi_{u,M}^{^{\prime}}(t_2)=0$ and
$\phi_{u,M}^{\prime}(t)<0$ for all $t\in (t_2,\infty)$ implies that
$\ds \phi_{u,M}(t_2)= \max_{t\geq t_*} \phi_{u,M}(t)$.\QED
\end{proof}

\begin{Lemma}\label{zle4}
If $\la$ be such that (\ref{zeq8}) holds. Then $\mc N_{\la,M}^{0}= \{0\}$.
\end{Lemma}

\proof Suppose $u\in \mc N_{\la,M}^{0}$, $u\not\equiv 0$. Then by
definition of $\mc N_{\la,M}^{0}$, we have the following two equations
\begin{align}
(2n-1)a\|u\|^{2n}+ (n-1)b\|u\|^n&=\int_{\Om}{g^{\prime}(u)u^2} dx+ \la q \int_{\Om}h(x)|u|^{q+1},\label{zeq10}\\
a\|u\|^{2n}+ b\|u\|^n &=\int_{\Om}g(u)u dx+ \la \int_{\Om}h(x)|u|^{q+1}.\label{zeq11}
\end{align}
Let $u\in H^+\cap \mc N_{\la,M}^{0}$ and $\la\in(0,\la_0)$. Then from
above equations, we can easily deduce that
\begin{align*}
(2n-1-q)a\|u\|^{2n}+ (n-1-q)b\|u\|^n&\leq \int_{\Om}{g^{\prime}(u)u^2} dx.
\end{align*}
Then using the inequality $\sqrt{ab} \le \frac{a+b}{2}$ for $a,b\ge 0$, we obtain
\[ 2\sqrt{(2n-1-q)ab}\|u\|^{\frac{3n}{2}}\leq (2n-1-q)a\|u\|^{2n}+ (n-1-q)b\|u\|^n.\]
Hence $u\in\La\setminus\{0\}$. Noting that
$g^\prime(s)s^2-(2n-1)g(s)s= (p+2-2n+\ba|s|^{\ba})|s|^{p+2}
e^{|s|^{\ba}}$, from \eqref{zeq10} and \eqref{zeq11}, we get
\begin{align*}
(2n-1-q)\la \int_{\Om}h(x)|u|^{q+1} &=
\int_{\Om}\left(p+2-2n+\ba|u|^{\ba}\right)|u|^{p+2}e^{|u|^{\ba}} dx + n b\|u\|^n\\
&> \int_{\Om}\left(p+2-2n+\ba|u|^{\ba}\right)|u|^{p+2}e^{|u|^{\ba}} dx,
\end{align*}
which violates Lemma \ref{zle3}. Hence $\mc N_{\la,M}^{0}= \{0\}$. In
other cases, $u\in H_{0}^-\cap\mc N_{\la, M}^{0}$, we see that $t=1$ is a critical point of $\phi_{u,M}(t)$ and
$\phi_{u,M}^{^{\prime\prime}}(1)=0$. But $u\in H^{-}_{0}$ implies that $\phi_{u,M}$ has exactly one critical point corresponding
to global maxima i.e $\phi_{u,M}^{\prime\prime}(1)\ne0$ which is a
contradiction. Hence $\mc N_{\la,M}^{0}= \{0\}$. \QED
\subsection{Existence and multiplicity of solutions}

\noi In this section we show that $J_{\la, M}$ is bounded below on $\mc
N_{\la,M}$. Also we show that $J_{\la,M}$
attains its minimizer on $H^{+}\cap \mc N_{\la,M}^{+}$.\\
\noi We define $\ds\theta_{\la,M} := \inf \left\{J_{\la,M}(u)\mid u\in
\mc N_{\la,M}\right\}$ and prove the following lower bound:
\begin{Theorem}\label{zth1}
$J_{\la,M}$ is bounded below and coercive on $\mc N_{\la,M}$. Moreover, there exists
a constant $C=C(p,q,n)>0$ such that $\theta_{\la,M} \geq - C
\la^{\frac{k}{k-1}}$.
\end{Theorem}
\proof Let $u \in \mc N_{\la,M}$. Then we have
\begin{align}\label{zeq4}
 J_{\la,M}(u)
&=\frac{(p+2-2n)}{2n(p+2)}a\|u\|^{2n}+
\frac{(p+2-n)}{n(p+2)}b\|u\|^{n} +
\int_{\Om}\left(\frac{1}{p+2}g(u)u
- G(u)\right)\notag\\
&\quad\quad-\frac{\la(p+1-q)}{(q+1)(p+2)}\int_{\Om}h|u|^{q+1}.
\end{align}
Using $G(s)\leq \frac{1}{p+2}g(s)s$ for all $s\in \mb R$,
H\"{o}lder's  and Sobolev inequalities in \eqref{zeq4}, we obtain
\begin{align*}
 J_{\la,M}(u)
&\geq \frac{(p+2-2n)}{2n(p+2)}a\|u\|^{2n}+\frac{(p+2-n)}{n(p+2)}b
\|u\|^{n}
-\frac{\la(p+1-q)}{(q+1)(p+2)}\int_{\Om}h(x)|u|^{q+1}dx\\
&\geq \frac{(p+2-2n)}{2n(p+2)}a\|u\|^{2n}+ \frac{(p+2-n)}{n(p+2)} b
\|u\|^{n} -\frac{\la(p+1-q)}{(q+1)(p+2)} C_0 \|u\|^{q+1},
\end{align*}
for some constant $C_0>0$, which shows $J_{\la,M}$ is coercive on
$\mc N_{\la,M}$ as $q+1< 2n$.

\noi Again for $u \in \mc N_{\la,M}$, we have
\begin{align}\label{zeq3}
J_{\la,M}(u) &= \frac{1}{2n} \int_{\Om}{g(u)u}- \int_{\Om} {G(u)}
-{\la}\left(\frac{1}{q+1}-\frac{1}{2n}\right) \int_{\Om}h(x)|u|^{q+1}+ \frac{b}{2n}\|u\|^n.
\end{align}
Also, It is easy to see that
\begin{equation}\label{1m}
\frac{1}{2n}g(u)u- G(u)\geq
\left(\frac{1}{2n}-\frac{1}{p+2}\right)|u|^{p+2+\ba} ,
\end{equation}
If $u\in H^{-}_{0}$, then $J_\la(u)$ is bounded below by $0$. If
$u\in H^+$ then by using H\"{o}lder's inequality, we have
\begin{equation*}
\int_{\Om}h(x)|u|^{q+1} \leq  l^{\frac{k-1}{k}} \left(\int_{\Om} |u|^{(q+1)k}
dx\right)^{\frac{1}{k}},
\end{equation*}
where $l=\int_{\Om} |h(x)|^{k/k-1} dx$.
From above inequalities, we get
\begin{align}
J_{\la,M}(u)&\geq \left(\frac{1}{2n}-\frac{1}{p+2}\right)\int_{\Om}
|u|^{(q+1)k} dx -
\frac{\la(2n-q-1)l^{\frac{k-1}{k}}}{2n(q+1)}\left(\int_{\Om}
|u|^{(q+1)k} dx\right)^{\frac{1}{k}},\nonumber
\end{align}
where $k=\frac{p+2+\ba}{q+1}$. By considering the global minimum of the function $\rho(x): \mb R^+
\lra \mb R$ defines as \\ $\rho(x)=
\left(\frac{1}{2n}-\frac{1}{p+2}\right)x^k -
\left(\frac{\la(2n-q-1)l^{\frac{k-1}{k}}}{2n(q+1)}\right) x,$ it can
be shown that
\begin{equation*}
\inf_{u\in \mc N_{\la,M}}J_{\la,M}(u) \geq
\rho\left[\left(\frac{\la(2n-q-1)(p+2)l^{\frac{k-1}{k}}}{k(q+1)(p+2-2n)}\right)^{\frac{1}{k-1}}\right].
\end{equation*}
From this it follows that
\begin{equation}\label{zeq5}
\theta_{\la,M} \geq -C(p,q,n)\la^{\frac{k}{k-1}},
\end{equation}
where $C(p,q,n)=
\left(\frac{1}{k^{\frac{1}{k-1}}}-\frac{1}{k^{\frac{k}{k-1}}}\right)
\frac{l(p+2)^{\frac{1}{k-1}}(2n-q-1)^{\frac{k}{k-1}}}{2n(p+2-2n)^{\frac{1}{k-1}}(q+1)^{\frac{k}{k-1}}}>0$.
Hence $J_{\la,M}$ is bounded below on $\mc N_{\la,M}.$ \QED

\noi The following lemma shows that minimizers for $J_{\la,M}$ on any
subset of $\mc N_{\la,M}$ are usually critical points for $J_{\la,M}$.
\begin{Lemma}\label{zle1}
Let $u$ be a local minimizer for $J_{\la,M}$ in any of the subsets of $\mc N_{\la,M}$ such that
$u\notin \mc N_{\la,M}^{0}$, then $u$ is a critical point
for $J_{\la,M}$.
\end{Lemma}
\proof Let $u$ be a local minimizer for $J_{\la,M}$ in any of the
subsets of $\mc N_{\la,M}$. Then, in any case $u$ is a
minimizer for $J_{\la,M}$ under the constraint $I_{\la,M}(u):=\ld
J_{\la,M}^{\prime}(u),u\rd =0$. Hence, by the theory of Lagrange
multipliers, there exists $\mu \in \mb R$ such that $ J
_{\la,M}^{\prime}(u)= \mu I_{\la,M}^{\prime}(u)$. Thus $\ld
J_{\la,M}^{\prime}(u),u\rd= \mu\;\ld I_{\la,M}^{\prime}(u),u\rd = \mu
\phi_{u,M}^{\prime\prime}(1)$=0, but $u\notin \mc N_{\la,M}^{0}$ and so
$\phi_{u,M}^{\prime\prime}(1) \ne 0$. Hence $\mu=0$ completes the
proof.\QED

\begin{Lemma}\label{zle31}
Let $\la$ satisfy (\ref{zeq8}).
Then given $u\in \mc N_{\la,M}\setminus \{0\}$, there exist $\e>0$
and a differentiable function $\xi : \textbf{B}(0,\e)\subset
W^{1,n}_{0}(\Om) \lra \mb{R}$ such that $\xi(0)=1$, the function
$\xi(w)(u-w)\in\mc N _{\la,M}$ and for all $w\in W^{1,n}_{0}(\Om)$
{\small \begin{align}\label{zeq:3.1}
\ld&\xi^{\prime}(0),w\rd=  \notag\\
&\frac{ n(a\|u\|^{n}+a+b)\int_{\Om}(|\na u|^{n-2}\na u
\na w-\ds\int_{\Om}\left({g(u)+g^{\prime}(u)u}\right)w - \la
(q+1)\ds\int_{\Om} h(x)|u|^{q-1}uw }{(2n-1-q)a\|u\|^{2n}+(n-q-1)b\|u\|^n-
\ds\int_{\Om}{g^{\prime}(u) u^2}dx+ q\ds\int_{\Om} g(u) u dx}.
\end{align}}
 \end{Lemma}
\proof Fix $u\in \mc N_{\la,M}\setminus \{0\}$, define a function
$G_u: \mb R\times W^{1,n}_{0}(\Om) \lra \mb R$ as follows:
{\small \begin{align*}
G_{u}(t,v)&= a t^{2n-1-q}\|u-w\|^{2n}+ bt^{n-1-q}\|u-v\|^{n}-  t^{-q}\int_{\Om} g(t(u-v)) (u-v)
dx -\la \int_{\Om}h|u-v|^{q+1}.
\end{align*}}
 Then $G_u\in C^{1}(\mb R\times W^{1,n}_{0}(\Om); \mb
R)$, $G_{u}(1,0)= \ld J_{\la,M}^{\prime}(u),u \rd= 0$ and
\begin{equation*}
\frac{\partial}{\partial t}G_{u}(1,0)= (2n-1-q)a \|u\|^{2n}+ (n-1-q)b \|u\|^n- \int_{\Om}
 g^{\prime}(u) u^2 dx +q \int_{\Om} g(u) u dx \ne 0,
\end{equation*}
since $\mc N_{\la,M}^{0}=\{0\}$. By the Implicit function
theorem, there exist $\e>0$ and a differentiable function $\xi :
\textbf{B}(0,\e)\subset W^{1,n}_{0}(\Om) \lra \mb{R}$ such that
$\xi(0)=1$, and $G_{u}(\xi(w),w)=0$ for all $w\in \textbf{B}(0,\e)$
which is equivalent to $\ld
J_{\la,M}^{\prime}(\xi(w)(u-w)),\xi(w)(u-w) \rd=0$ for all $w\in
\textbf{B}(0,\e)$ and hence $\xi(w)(u-w)\in \mc N _{\la,M}$. Now
differentiating $G_{u}(\xi(w),w)=0$ with respect to $w$ we obtain
\eqref{zeq:3.1}. \QED

\begin{Lemma}\label{zle32}
Let $\la$ satisfy (\ref{zeq8}).
Then given $u\in \mc N_{\la,M}^{-}\setminus \{0\}$, there exist $\e>0$
and a differentiable function $\xi^{-} : \textbf{B}(0,\e)\subset
W^{1,n}_{0}(\Om) \lra \mb{R}$ such that $\xi^{-}(0)=1$, the function
$\xi^{-}(w)(u-w)\in\mc N _{\la,M}$ and for all $w\in W^{1,n}_{0}(\Om)$
\begin{align*}
\ld&(\xi^{-})^{\prime}(0),w\rd=  \notag\\
&\frac{ n(a\|u\|^{n}+a+b)\int_{\Om}(|\na u|^{n-2}\na u
\na w-\ds\int_{\Om}\left({g(u)+g^{\prime}(u)u}\right)w - \la
(q+1)\ds\int_{\Om} h(x)|u|^{q-1}uw }{(2n-1-q)a\|u\|^{2n}+(n-q-1)b\|u\|^n-
\ds\int_{\Om}{g^{\prime}(u) u^2}dx+ q\ds\int_{\Om} g(u) u dx}.
\end{align*}
 \end{Lemma}
\proof First, we note that if $u\in\mc N_{\la,M}^{-}$, then
$u\in\La\setminus\{0\}$, satisfies (\ref{zeq8}). Then Lemma \ref{zle31}, there
exist $\e>0$ and a differentiable function $\xi^{-} :
\textbf{B}(0,\e)\subset W^{1,n}_{0}(\Om) \lra \mb{R}$ such that
$\xi^{-}(0)=1$ and the function $\xi^{-}(w)(u-w)\in \mc N _{\la,M}$
for all $w\in \textbf{B}(0,\e)$. Since $u\in \mc N_{\la,M}^{-}$, we
have
\begin{equation*}
(2n-1-q) a\|u\|^{2n}+ (n-1-q) b\|u\|^{n} + q \int_{\Om} g(u)u dx - \int_{\Om}{g^{\prime}(u)u^2}dx <0.
\end{equation*}
Thus by continuity of  $J_{\la,M}^{^{\prime}}$ and $\xi^{-}$, we have
\begin{align*}
\phi^{\prime\prime}&_{(\xi^{-}(w)(u-w),M)}(1) = (2n-1-q)
a \|\xi^{-}(w)(u-w)\|^{2n}+ (n-1-q)b\|\xi^{-}(w)(u-w)\|^{n}\\\notag
&+q \int_{\Om}
g(\xi^{-}(w)(u-w))\xi^{-}(w)(u-w) - \int_{\Om}
g^{\prime}(\xi^{-}(w)(u-w))(\xi^{-}(w)(u-w))^2 <0,\notag
\end{align*}
if $\e$ is sufficiently small. This concludes the proof.

\begin{Lemma}\label{zle33}
 There exists a constant $C_2>0$ such that $\theta_{\la,M}\leq
-\frac{(p+1-q)}{n(q+1)(p+2)}C_2$.
\end{Lemma}
\proof Let $v$ be such that $\int_{\Om}h|v|^{q+1}>0$. Then by the fibering map
analysis, we can find $t_{1}=t_{1}(v)>0$ such that $t_{1}v\in \mc
N_{\la,M}^+$. Thus
\begin{align}\label{zeq31}
&J_{\la,M}(t_{1}v)=\left(\frac{1}{2n}-\frac{1}{q+1}\right)a\|t_{1}v\|^{2n} + \left(\frac{1}{n}-\frac{1}{q+1}\right)b\|t_{1}v\|^{n}-
\int_{\Om}G(t_{1}v) +\frac{1}{q+1} \int_{\Om} g(t_{1}v) t_{1}v \notag\\
&\leq \frac{2n+q}{2n(q+1)}\int_{\Om}g(t_{1}v) t_{1}vdx -\int_{\Om}
G(t_{1}v)dx -\frac{1}{2n(q+1)} \int_{\Om} g^{\prime}(t_{1}v) (t_{1}v)^2dx,
\end{align}
since $t_{1}v\in \mc N_{\la,M}^{+}$. We now consider the
function
\begin{equation*}
\rho(s)= \frac{2n+q}{2n(q+1)}g(s)s- G(s)-\frac{1}{2n(q+1)}g^{\prime}(s)s^2.
\end{equation*}
Then
\begin{align*}
\rho^{\prime}(s)&=
\frac{(q+2n-2)}{2n(q+1)}g^{\prime}(s)s-\frac{q(2n-1)}{2n(q+1)}g(s)-\frac{1}{2n(q+1)}g^{\prime\prime}(s)s^2\\
&= \left(\frac{(q+2n-2-p)(p+1)-(n-1)q}{2n(q+1)} \right) g(s)\\
& \quad \quad + \ba \left( \frac{q-p+2n-2-\ba - p-1}{2n(q+1)} \right)
g(s)|s|^{\ba} - \frac{\ba^2}{2n(q+1)}
 g(s)|s|^{2\ba}.
\end{align*}
Now it is not difficult to see that coefficients in the first and
second term are negative, since $p>2n-2$. As $\rho(0)=0$, it follows
that $\rho(s)\leq 0$ for all $s\in\mb R^{+}$. Also it can be easily
verified that
\[\lim_{s\ra 0}\frac{\rho(s)}{|s|^{p+2}}=
-\frac{(p+1-q)(p+2-2n)}{2n(q+1)(p+2)},\;\quad \lim_{s\ra \infty}\frac{\rho(s)}{|s|^{p+2+\ba} e^{|s|^{\ba}}}=
-\frac{\ba}{2n(q+1)}.\]
From these two estimates, we get that
\begin{equation}\label{zeq30}
\rho(s)\leq-\frac{(p+1-q)}{2n(q+1)(p+2)}\left(p+2-2n+\ba|s|^{\ba}\right)
|s|^{p+2}e^{|s|^{\ba}}.
\end{equation}
Therefore, using (\ref{zeq31}) and (\ref{zeq30}), we get
\begin{align*}
J_{\la,M}(t_{1}v)&\leq - \frac{(p+1-q)}{2n(q+1)(p+2)}\int_{\Om}
\left(p+2-2n+\ba|t_{1}v|^{\ba}\right) |t_{1}v|^{p+2} ~
e^{|t_{1}v|^{\ba}} dx \notag\\
& \leq - \frac{(p+1-q)}{2n(q+1)(p+2)}\int_{\Om}|t_{1}v|^{p+2+\ba} dx
\end{align*}
Hence $\theta_{\la,M}\leq\inf_{u\in \mc N_{\la,M}^{+}\cap
H^{+}}J_{\la,M}(u)\leq  -\frac{(p+1-q)}{2n(q+1)(p+2)}\; C_2$, where
$C_2=\int_{\Om} |t_{1}v|^{p+2+\ba}dx.$\QED

\noi By Lemma \ref{zth1}, $J_{\la,M}$ is bounded below on $\mc
N_{\la,M}$. So, by Ekeland's Variational principle, we can find a
sequence $\{u_k\}\in \mc N_{\la,M}\setminus \{0\}$ such that
\begin{align}
J_{\la,M}(u_k)&\leq \theta_{\la,M}+\frac{1}{k},\label{zeq32}\\
J_{\la,M}(v)&\geq J_{\la,M}(u_k)- \frac{1}{k}\|v-u_k\|\;\;\mbox{for all}\;\;v\in \mc
N_{\la,M}.\label{zeq33}
\end{align}
\noi Now from (\ref{zeq32}) and Lemma \ref{zle33},
we have
\begin{align}\label{zeq34}
 J_{\la,M}(u_k)&\leq -\frac{(p+1-q)}{2n(q+1)(p+2)}\; C_3.
\end{align}
\noi Also as $u_k\in \mc N_{\la,M}$, we have
\begin{align*}
J_{\la,M}(u_k) &=\left(\frac{1}{2n}-\frac{1}{p+2}\right)a \|u_k\|^{2n}+ \left(\frac{1}{n}-\frac{1}{p+2}\right)b \|u_k\|^{n}
-\frac{\la(p+1-q)}{(q+1)(p+2)} \int_{\Om}h|u_k|^{q+1}\\
&\quad\quad+\int_{\Om}\left(\frac{1}{p+2}g(u_k)u_k - G(u_k)\right) dx.
\end{align*}

\noi This together with \eqref{zeq34} and $\frac{1}{p+2}g(u_k)u_k - G(u_k)\geq0$, we obtain
\begin{align}\label{zeq35}
H(u_k) \geq \frac{C_3}{2n \la }>0\;\mbox{for all}\;k.
\end{align}

\noi Thus we have $u_k\in \mc N_{\la,M}\cap  H^{+}$. Now we
prove the following:
\begin{Proposition}\label{zpro1}
Let $\la$ satisfies (\ref{zeq8}).
Then $\|J_{\la,M}^{\prime}(u_k)\|_{*}\ra 0$ as $k\ra \infty$.
\end{Proposition}
\proof Step 1: $\liminf_{k\ra \infty} \|u_k\|>0$.\\
\noi Applying H\"{o}lder's inequality in (\ref{zeq35}), we have
$K'\|u_k\|^{q+1}\geq \int_{\Om} h|u_k|^{q+1} \geq \frac{C_3}{2n \la }>0$ which implies
that $\ds \liminf_{k\ra \infty} \|u_k\|>0$. \\
Step 2: We claim that
{\small\begin{equation}\label{zliminf3.4}
\ds K:=\liminf_{k \ra \infty}\left\{(2n-1-q) a\|u_k\|^{2n}+ (n-1-q)b\|u_k\|^{n} - \int_{\Om}
g^{\prime}(u_k) {u^2_k} dx + q \int_{\Om} g(u_k) u_k dx\right\}>0.
\end{equation}}
Assume by contradiction that for some subsequence of $\{u_k\}$,
still denoted by $\{u_k\}$ we have
\begin{align*}
(2n-1-q) a\|u_k\|^{2n}+ (n-1-q)b\|u_k\|^{n} - \int_{\Om} g^{\prime}(u_k) u_k^2 dx + q
\int_{\Om}g(u_k) u_k dx= o_{k}(1),
\end{align*}
where $o_k(1)\ra 0$ as $k\ra\infty$. From this and the fact that $\{u_k\}$ is bounded away from $0$,
we obtain that $\ds \liminf_{k\ra \infty}\int_{\Om}{g^{\prime}(u_k) u_k^2}
dx >0.$ Hence, we get $u_k \in \La\setminus \{0\}$ for all $k$ large.
Using this and the fact that $u_k \in \mc N_{\la,M}\setminus \{0\}$,
we have
\begin{align*}
o_{k}(1)= (2n-q-1)\la\int_{\Om} h|u_k|^{q+1}- nb \|u_k\|^n
-\int_{\Om}(g^{\prime}(u_k) u_k^2 -(2n-1)g(u_k)u_k) dx < -\La_m
\end{align*}
by (\ref{zeq8}), which is a contradiction.

\noi Finally, we show that $\|J_{\la,M}^{^{\prime}}(u_k)\|_{*}\ra 0$ as $k\ra
\infty$. By Lemma \ref{zle31}, we obtain a sequence of functions
$\xi_k :\textbf{B}(0,\e_k)\ra \mb R$ for some $\e_k>0$ such that
$\xi_k(0)=1$ and $\xi_k(w)(u_k-w)\in \mc N_{\la,M}$ for all
$w\in\textbf{B}(0,\e_k)$. Choose $0<\rho <\e_k$ and $f\in
W^{1,n}_{0}(\Om) $ such that $\|f\|=1$. Let $w_\rho=\rho f$. Then
$\|w_{\rho}\|=\rho<\e_k$ and
$\eta_{\rho}=\xi_k(w_\rho)(u_k-w_\rho)\in\mc N_{\la,M}$ for all $k$.
Since $\eta_{\rho} \in\mc N_{\la,M}$, we deduce from (\ref{zeq33}) and
Taylor's expansion,
\begin{align}\label{zeq:3.7}
\frac{1}{n}\|\eta_{\rho}-u_k\|&\geq J_{\la,M}(u_k)-J_{\la,M}(\eta_\rho)
=\ld J_{\la,M}^{^{\prime}}(\eta_\rho),u_k-\eta_\rho\rd + o(\|u_k-\eta_\rho\|)\notag\\
&= (1-\xi_k(w_{\rho}))\ld J_{\la,M}^{\prime}(\eta_\rho),u_k\rd + \rho
\xi_k(w_{\rho})\ld J_{\la,M}^{\prime}(\eta_\rho),f\rd +
o(\|u_k-\eta_\rho\|).
\end{align}
We note that as $\rho \ra 0$, $\frac{1}{\rho}\|\eta_\rho -u_k\|=
\|u_k\ld\xi_{k}^{\prime}(0),f\rd -f\|.$ Now dividing \eqref{zeq:3.7}
by $\rho$ and taking limit $\rho \ra 0$, and using $u_k\in \mc
N_{\la,M}$, we get
\begin{equation}\label{zeq36}
\ld J_{\la,M}^{\prime}(u_k),f \rd\leq
\frac{1}{k}\left(\|u_k\|\|\xi^{\prime}_{k}(0)\|_{*}+1\right)\le
\frac{1}{k}\frac{C_4\|f\|}{K},
\end{equation}
by Lemma \ref{zle31} and \eqref{zliminf3.4}. This completes the proof
of Proposition.\QED

\noi We can now prove the following:
\begin{Lemma}\label{zle34}
Let $\ba<\frac{n}{n-1}$ and let $\la$ satisfy (\ref{zeq8}). Then there exists a function $u_{\la}\in \mc N_{\la,M}^{+}\cap H^+$ such that $\ds J_{\la,M}(u_{\la})=\inf_{u\in \mc N_{\la,M}\setminus\{0\}}J_{\la,M}(u)$.
\end{Lemma}
\proof Let $\{u_k\}$ be a minimizing sequence for $J_{\la,M}$ on $\mc
N_{\la,M}\setminus \{0\}$ satisfying (\ref{zeq32}) and (\ref{zeq33}).
Then $\{u_k\}$ is bounded in $W_{0}^{1,n}(\Om)$. Also there exists a subsequence of $\{u_k\}$ (still denoted by
$\{u_k\}$) and a function $u_\la$ such that $u_k \rightharpoonup
u_\la$ weakly in $W_{0}^{1,n}(\Om)$, $u_k\ra u_\la$ strongly in
$L^{\al}(\Om)$ for all $\al\geq 1$ and $u_k(x)\ra u_{\la}(x)$ a.e in
$\Om$. Also $\int_{\Om} h|u_k|^{q+1} \ra \int_{\Om} h|u_\la|^{q+1}$ and by the compactness of Moser-Trudinger imbedding for $\beta<\frac{n}{n-1}$, $\int_{\Om} f(u_k)(u_k-u_\la)\ra 0$ as $k\ra\infty$. Then by Lemma \ref{zpro1}, we have $J_{\la,M}^{\prime}(u_k-u_\la) \ra 0$. We conclude that
 \[ m(\|u_k\|^n)\int_{\Om}|\na u_k|^{n-2}\na u_k(\na u_k-\na u_\la)\ra 0.\]
On the other hand, using $u_k\rightharpoonup u_{\la}$ weakly and boundedness of $m(\|u_k\|^n)$,
{\small\[m(\|u_k\|^n) \int_{\Om} |\na u_\la|^{n-2} \na u_\la (\na u_k - \na u_\la)\ra 0\; \mbox{as}\; k\ra\infty.\]}
From above two equations and inequality \eqref{11e2}, we have
\[m(\|u_k\|^n)\int_{\Om}|\na u_k- \na u_\la|^n \ra 0\;\mbox{as}\; k\ra\infty.\]
Since $m(t)\ge m_0$, we obtain $u_k\ra u_{\la}$ strongly in $W^{1,n}_{0}(\Om)$ and hence $\|u_k\| \rightarrow \|u_\la\|$ strongly as $k\ra\infty$. In particular, it follows that $u_{\la}$ solves $(P_{\la, M})$ and hence $u_\la \in \mc N_{\la,M}$. Moreover,
$\ds\theta_\la \leq J_{\la,M}(u_\la)\leq \liminf_{k\ra
\infty}J_{\la,M}(u_k)=\theta_\la$. Hence $u_\la$ is a minimizer for
$J_{\la,M}$ on $\mc N_{\la,M}$.\\
\noi Using (\ref{zeq35}), we have $\int_{\Om} h|u_\la|^{q+1}>0$. Therefore there exists $t_{1}(u_\la)$ such that
$t_{1}(u_{\la})u_\la \in \mc N_{\la,M}^{+}$. We now claim that
$t_{1}(u_\la)=1$ $(i.e.$ $u_\la \in \mc N_{\la,M}^{+})$. Suppose
$t_{1}(u_\la)<1$. Then $t_{2}(u_\la)=1$ and hence $u_\la \in \mc
N_{\la,M}^{-}$. Now $J_{\la,M}(t_{1}(u_\la)u_\la)\leq J_{\la,M}(u_\la)=
\theta_{\la,M}$ which is impossible, as $t_{1}(u_\la) u_\la \in \mc
N_{\la,M}$.\QED

\begin{Theorem}\label{zth3.6}
Let $\ba<\frac{n}{n-1}$ and let $\la$ be such that (\ref{zeq8}) holds.
Then $u_\la \in \mc N_{\la,M}^{+}\cap H^+$ is also a
non-negative local minimum for $J_{\la,M}$ in $W^{1,n}_{0}(\Om)$.
\end{Theorem}
\proof Since $u_\la \in \mc N^{+}_{\la,M}$, we have $t_{1}(u_{\la})=1
<t_*(u_{\la})$. Hence by continuity of $u\mapsto t_*(u)$, given
$\e>0$, there exists $\de=\de(\e)>0$ such that $1+\e< t_*(u_\la-w)$
for all $\|w\|<\de$. Also, from Lemma \ref{zle33} we have, for $\de>0$
small enough, we obtain a $C^1$ map $t: \textbf{B}(0,\de)\lra \mb R^+$
such that $t(w)(u_\la-w)\in \mc N_{\la,M}$, $t(0)=1$. Therefore, for
$\de>0$ small enough we have $t_{1}(u_\la-w)=
t(w)<1+\e<t_*(u_\la-w)$ for all $\|w\|<\de$. Since $t_*(u_\la-w)>1$,
we obtain $J_{\la,M}(u_\la)<J_{\la,M}(t_{1}(u_\la-w)(u_\la-w))<J_{\la,M}(u_\la-w)$
for all $\|w\|<\de$. This shows that $u_\la$ is a local minimizer for
$J_{\la,M}$.\\
\noi Now we show that $u_\la$ is a non-negative local minimum for
$J_{\la,M}$ on $W^{1,n}_{0}(\Om)$. If $u_\la\geq 0$ then we are done,
otherwise, if $u_{\la}\not\geq 0$ then we take $\tilde{u_{\la}}=
t_1(|u_\la|)|u_{\la}|$ which is non negative function in $\mc
N_{\la,M}^{+}\cap H^{+}$. As $\psi_{u_{\la}}(t)= \psi_{|u_{\la}|}(t)$ so
$t_{*}(|u_{\la}|)=t_{*}(u_{\la})$ and $t_{1}(u_{\la})\leq
t_1(|u_\la|)$. Hence $ t_1(|u_\la|)\geq 1$. Also $|u_{\la}| \in H^+$
then from Lemma \ref{zle5} we have $J_{\la,M}(\tilde{u_{\la}}) \leq
J_{\la,M}(|u_{\la}|)\leq J_{\la,M}(u_{\la})$. Hence $\tilde{u_{\la}}$
minimize $J_{\la,M}$ on $\mc N_{\la,M}\setminus \{0\}$. Thus we can
proceed same as earlier to show that $\tilde{u_{\la}}$ is a local
minimum for $J_{\la,M}$ on $W^{1,n}_{0}(\Om).$\QED

\begin{Lemma}\label{zle03} Let $\beta<\frac{n}{n-1}$ and let $\la$ be such that \eqref{zeq8} holds. Then
$J_{\la,M}$ achieve its minimizers on $\mc N_{\la,M}^{-}$.
\end{Lemma}
\proof
We note that $\mc N_{\la,M}^{-}$ is a closed set, as $t^{-}(u)$
is a continuous function of $u$ and  ${J}_{\la,M}$ is bounded
below on $\mc N_{\la,M}^{-}$. Therefore, by Ekeland's Variational
principle, we can find a sequence $\{v_k\}\in \mc N_{\la,M}^{-}$ such
that
{\small\begin{align*}\label{zeq46}
{J}_{\la,M}(v_k)&\leq \inf_{u\in \mc
N_{\la,M}^{-}}{J}_{\la,M}(u) +\frac{1}{k},\;\;
{J}_{\la,M}(v)\geq{J}_{\la,M}(v_k)-\frac{1}{k}\|v-v_k\|
\;\mbox{for all}\; v\in \mc N_{\la,M}^{-}.
\end{align*}}
Then $\{v_k\}$ is a bounded sequence in $W^{1,n}_{0}(\Om)$ and
is easy to see that $v_k\in
\La\setminus \{0\}$. Thus by Lemma \ref{zle32}
and following the proof of Lemma \ref{zpro1}, we get
$\|{J}_{\la,M}^{\prime}(v_k)\|_{*}\ra 0$ as $k\ra \infty$. Thus
following the proof as in Lemma \ref{zle34}, we have $v_\la\in\mc N_{\la, M}^{-}$, weak
limit of sequence $\{v_k\}$, is a solution of $(P_{\la, M})$. And moreover $v_\la \not\equiv 0$, as $\mc
N^{0}_{\la,M}=\{0\}$.\QED

\noi{\bf Proof of Theorem \ref{zth11}: } Now the proof follows from Lemmas \ref{zle34} and \ref{zle03}.\QED

\noi To obtain the existence result in the critical case, we need the following compactness Lemma.
\begin{Lemma}\label{cpt}
Suppose $\{u_k\}$ be a sequence in $W^{1,n}_{0}(\Om)$ such that
\[ J_{\la,M}^{\prime}(u_k) \rightarrow 0\;\; J_{\la,M}(u_k) \rightarrow c < \frac{1}{2n}m_0\al_{n}^{n-1}- C \la^{\frac{p+2+\ba}{p+1-q+\ba}},\] where $C$ is a positive constant depending on $p$, $q$ and $n$.
Then there exists a strongly convergent subsequence.
\end{Lemma}
\proof
By Lemma \ref{lem714}, there exists a subsequence $\{u_k\}$ of $\{u_k\}$ such that $u_k \rightarrow u$ in $L^\al(\Om)$ for all $\al$, $u_k(x) \rightarrow u(x)$ a.e. in $\Om$, $\na u_k(x) \rightarrow \na u(x)$ a.e. in $\Om$ and $|\na u_k|^{n-2}\na u_k \rightharpoonup |\na u|^{n-2} \na u$ weakly in $W^{1,n}_{0}(\Om)$. Now by concentration compactness lemma, $|\na u_k|^n \rightarrow \mu_1$, $g(u_k)u_k \rightarrow \mu_2$ in measure.

\noi Let $B=\{ x\in \overline{\Om}: \; \exists\; r=r(x),\;
\mu_1(\textbf{B}_r\cap \Omega) < (\al_n)^{n-1}\}$ and let
$A=\overline{\Om}\backslash B.$ Then as in Lemma \ref{lem714}, we
can show that $A$ is finite set say $\{x_1, x_2,...x_m\}.$ Since
$J_{\la,M}^{\prime} (u_k)\rightarrow 0$, we have
{\small\begin{align}
0=&\lim_{k\ra\infty} \langle J^{\prime}_{\la,M}(u_k), \phi\rangle = \lim_{k\ra\infty} m (\|u_k\|^n) \int_\Om |\na u_k|^{n-2} \na u_k \na \phi -\la \int_\Om |u_k|^{q-1} u_k \phi- \int_\Om g(u_k) \phi \label{a1}\\
0=&\lim_{k\ra\infty}\langle J^{\prime}_{\la,M}(u_k), u_k \phi\rangle = \lim_{k\ra\infty}m(\|u_k\|^n) \left(\int_\Om |\na u_k|^{n-2} \na u_k \na \phi u_k + \int_\Om |\na u_k|^n  \phi \right) \notag\\ &\hspace*{4cm}-\la\int_\Om |u|^{q+1}\phi- \lim_{k\ra\infty}\int_\Om g(u_k) u_k \phi \label{a2}\\
0=&\lim_{k\ra\infty} \langle J^{\prime}_{\la,M} (u_k), u\phi\rangle= \lim_{k\ra\infty} m(\|u_k\|^n)\left(\int_\Om |\na u_k|^{n-2} \na u_k \na u\phi+\int_\Om |\na u_k|^{n-2}\na u_k \na \phi u\right)\notag\\
&\hspace*{4cm}-\la \int_\Om |u|^{q+1} \phi -\int_{\Om} g(u) u \phi
\label{a3}
\end{align}}
Substituting \eqref{a3} in \eqref{a2}, we have
{\small \begin{equation}\label{a4}
\int_\Om g(u_k) u_k \phi = m(\|u_k\|^n)\int_\Om( |\na u_k|^n -  |\na u_k|^{n-2} \na u_k \na u) \phi +\int_\Om g(u) u \phi
\end{equation}}
Now take cut-off function $\psi_\de \in C^{\infty}_{0}(\Om)$ such that $\psi_\de(x)\equiv1$ in $\textbf{B}_\de(x_j),$ and $\psi_\de (x)\equiv0$ in $\textbf{B}_{2\de}^{c}(x_j)$ with $|\psi_\de|\le 1$. Then taking $\phi=\psi_\de$,
{\small\[0\le \left|\int_\Om |\na u_k|^{n-2} \na u_k \na u \phi\right| \le \left(\int_{\Om} |\na u_k|^{n}\right)^{(n-1)/n} \left(\int_{\textbf{B}_{2\de}} |\na u|^n\right)^{1/n} \rightarrow 0 \;\; \text{as}\;\; \de \rightarrow 0.\]}
Hence from \eqref{a4}, we get
\begin{equation}\label{a5}
\int_\Om \phi d \mu_2 \ge m_0 \int_\Om \phi d\mu_1 + \int_\Om g(u)u \phi\;\mbox{as}\; \de\ra 0.
\end{equation}
Now as in Lemma \ref{lem714}, we can show that for any relatively compact set $K\subset \Om_\e$, where $\Om_\e = \Om\backslash \cup_{i=1}^{m} \textbf{B}_\de(x_i)$
\[\lim_{k\rightarrow \infty} \int_K g(u_k)u_k \rightarrow \int_K g(u) u.\]
Also taking  $0<\e<\e_0$ and $\phi\in C^\infty_c(\mb R^n)$ such that
$\phi\equiv1$ in $\textbf{B}_{1/2}(0)$ and $\phi\equiv0$ in $\bar\Om\setminus \textbf{B}_1(0)$.
Take $\psi_\e=1-\ds\sum_{j=1}^m\phi\left(\frac{x-x_j}{\e}\right)$ in \eqref{a5}. Then $0\leq\psi_\e\leq 1$, $\psi_\e\equiv 1$ in
 $\bar\Om_\e=\bar\Om\setminus\cup_{j=1}^m \textbf{B}_\e(x_j)$,$\psi_\e\equiv0$ in $\cup_{j=1}^m \textbf{B}_{\e/2}(x_j)$
{\small\begin{align*}
\int_\Om \psi_\e d \mu_2 & =\lim_{\e\rightarrow 0}\left(\int_{\Om_\e} \psi_\e d \mu_2 + \sum_{i=1}^{m} \int_{\textbf{B}_{\e}\cap \Om} \psi_\e d\mu_2\right)
=\lim_{\e \rightarrow 0}\int_{\Om_\e} g(u) u \psi_\e + \sum_{i=1}^{m} \ba_i \de_{x_i}\\
&=\int_\Om g(u) u +\sum_{i=1}^{m} \ba_i \de_{x_i}.
\end{align*}}
Therefore, from \eqref{a5}, we get
\begin{equation}\label{a6}
m_0 \int_\Om \psi_\e d\mu_1 \le \sum_{i=1}^{m} \ba_i \de_{x_i}.
\end{equation}
Now choosing $\e \rightarrow 0$, we get
\[m_0\mu_1(A) \le \sum_{i=1}^{m} \ba_i .\]
Therefore from the definition of $A$, either $\ba_i=0$ or $\ba_i \ge m_0 (\al_n)^{n-1}$.
Now we will show that $\ba_i=0,$ for all $i$. Suppose not, Now using
$J_{\la,M}(u_k)\rightarrow c$ implies {\small\begin{align*}
nc= &J_{\la,M}(u_k)- \frac{1}{2} \langle J_{\la,M}^{\prime} (u_k)u_k\rangle \\
 = &\left(M(\|u_k\|^n)-\frac12 m(\|u_k\|^n)\|u_k\|^n\right)+ \int_\Om \left(\frac12 g(u_k) u_k -n G(u_k)\right)
 \\
 &+\la\left(\frac{1}{2}-\frac{n}{q+1}\right)\int_\Om h |u|^{q+1}\\
\ge & \frac{ m_0(\al_n)^{n-1}}{2}+\int_\Om \left(\frac{1}{2} g(u)u-nG(u)\right) + \la\left(\frac{1}{2}-\frac{n}{q+1}\right)   \int_\Om h|u|^{q+1}.
\end{align*}}
\noi Then using equation \eqref{1m}, we have
\begin{align*}
c\ge &\frac{1}{2n}m_0(\al_n)^{n-1} + \left(\frac{1}{2n}-\frac{1}{p+2}\right)\int_{\Om} |u|^{p+2+\ba} + \la \left(\frac{1}{2n}-\frac{1}{q+1}\right)\int_\Om h |u|^{q+1}\\
\geq & \frac{1}{2n}m_0(\al_n)^{n-1} +
\left(\frac{1}{2n}-\frac{1}{p+2}\right)\int_{\Om} |u|^{(q+1)k} -
\frac{\la(2n-1-q)l^{\frac{k-1}{k}}}{2n(q+1)}\left(\int_\Om
|u|^{(q+1)k}\right)^{\frac{1}{k}},
\end{align*}
where $k=\frac{p+1+\ba}{q+1}$. Now as in Theorem \ref{zth1}, consider the global minimum of the function $\rho(x): \mb R^+
\lra \mb R$ defines as \\ $$\rho(x)=
\left(\frac{1}{2n}-\frac{1}{p+2}\right)x^k -
\left(\frac{\la(2n-q-1)l^{\frac{k-1}{k}}}{2n(q+1)}\right) x.$$
Then it can be shown that $\rho$ attains its minimum value at $x=\left(\frac{\la(2n-q-1)(p+2)l^{\frac{k-1}{k}}}{k(q+1)(p+2-2n)}\right)^{\frac{1}{k-1}}$
and its minimum value is  {\small$-C(p,q,n)\la^{\frac{k}{k-1}}$, where $C(p,q,n)=
\left(\frac{1}{k^{\frac{1}{k-1}}}-\frac{1}{k^{\frac{k}{k-1}}}\right)
\frac{l(p+2)^{\frac{1}{k-1}}(2n-q-1)^{\frac{k}{k-1}}}{2n(p+2-2n)^{\frac{1}{k-1}}(q+1)^{\frac{k}{k-1}}}>0.$}
Therefore, $c\ge \frac{1}{2n}m_0(\al_n)^{n-1} - C(p,q,n)\la^{\frac{p+2+\ba}{p+1-q+\ba}}$ .\QED


\noi Let $\la_{00}=\max \{ \la: \; \theta_{\la,M} \le \frac{1}{2n}m_0\al_{n}^{n-1}- C \la^{\frac{p+2+\ba}{p+1-q+\ba}}\}$ where $C$ is as in the above Lemma.\\

\noi{\bf Proof of Theorem \ref{zth12}:} Let $\{u_k\}$ be a minimizing sequence for $J_{\la,M}$ on $\mc
N_{\la,M}\setminus \{0\}$ satisfying (\ref{zeq33}).
Then it is easy to see that $\{u_k\}$ is a bounded sequence in $W_{0}^{1,n}(\Om)$.
Also there exists a subsequence of $\{u_k\}$ (still denoted by
$\{u_k\}$) and a function $u_\la$ such that $u_k \rightharpoonup
u_\la$ weakly in $W_{0}^{1,n}(\Om)$, $u_k\ra u_\la$ strongly in
$L^{\al}(\Om)$ for all $\al\geq 1$ and $u_k(x)\ra u_{\la}(x)$ a.e in
$\Om$. Then by Lemma \ref{zpro1}, we have $J_{\la,M}^{\prime}(u_k-u_\la) \ra 0$.

Now by compactness Lemma \ref{cpt}, $u_k \rightarrow u_\la$ strongly in $W^{1,n}_{0}(\Om)$ and hence $\|u_k\|\rightarrow \|u_\la\|$ strongly as $k\rightarrow \infty$. In particular, it follows that $u_\la$ solves $(P_{\la, M})$ and hence $u_\la \in \mc N_{\la, M}.$ Also we can show similarly as in Lemma \ref{zle34} and Theorem \ref{zth3.6} that $u_{\la} \in \mc N_{\la,M}^{+}\cap H^{+}$ is a non-negative local minimizer of $J_{\la,M}$ in $W^{1,n}_{0}(\Om)$.\QED

\noindent{\bf Acknowledgements:} Third author's research is supported by National Board for Higher Mathematics, Govt. of India, grant number: 2/48(12)/2012/NBHM(R.P.)/R$\&$D II/13095.


\begin{thebibliography}{21}
\footnotesize
\bibitem{A} Adimurthi, {\it Existence of positive solutions of the semilinear Dirichlet problem with critical growth for the $n$-Laplacian},
Ann. Scuola Norm. Sup. Pisa Cl. Sci., 17 (1990) 393-413.
\bibitem{AS} Adimurthi and K. Sandeep, {\it A singular Moser-Trudinger embedding and its applications,} NoDEA Nonlinear Differential Equations Appl. 13 (2007) 585-603.
\bibitem{AG} C. O. Alves, F. J. S. A. Corr$\hat{e}$a and G. M. Figueiredo, {\it On a class of nonlocal elliptic problems with critical growth}, DEA 2 (2010) 409-417.

\bibitem{AF} C. O. Alves, F. J. S. A. Corr$\hat{e}$a and T. F. Ma, {\it Positive solutions for a quasilinear elliptic equation of Kirchhoff type}, Comput. Math. Appl. 49 (2005) 85-93.
\bibitem{COA} C. O. Alves and A. El Hamidi, {\it Nehari manifold and
existence of positive solutions to a class of quasilinear problem},
Nonlinear Anal. 60 (4) (2005) 611-624.

\bibitem {ABC} A. Ambrosetti, H. Brezis and G. Cerami,
{\it Combined effects of concave and convex nonlinearities in some
elliptic problems}, J. Funct. Anal. 122 (2) (1994) 519-543.
\bibitem{KY} K. J. Brown and Y. Zhang, {\it The Nehari manifold for a semilinear elliptic problem with a
sign-changing weight function}, J. Differential Equations, 193
(2003) 481-499.

\bibitem{WUFI} K. J. Brown and T. F. Wu, {\it A fibering map approach to a semilinear
elliptic boundry value problem}, Electronic Journal of Differential
Equations 69 (2007) 1-9.

\bibitem{CK} C. Chen, Y. Kuo and T. Wu, {\it The Nehari manifold for a Kirchhoff type problem involving sign-changing weight functions,} J. Differ. Equ. 250 (4) (2011) 1876-1908.

\bibitem{CX} B. T. Cheng, X. Wu and J. Liu, {\it Multiple solutions for a class of Kirchhoff type problems with concave nonlinearity,} Nonlinear Differ. Equ. Appl. 19 (5) (2012) 521-537.
\bibitem{BW} B. T. Cheng and X. Wu, {\it Existence results of positive solutions of Kirchhoff problems}, Nonlinear Anal. 71 (2009) 4883-4892.
\bibitem{CH} C. S. Chen, J. C. Huang and L. H. Liu, {\it Multiple solutions to the nonhomogeneous $p$-Kirchhoff elliptic equation with
concave-convex nonlinearities,} Appl. Math. Lett. 26 (7) (2013) 754-759.
\bibitem{chipot} M. Chipot and B. Lovat, {\it Some remarks on nonlocal elliptic and parabolic problems,} Nonlinear Anal. 30 (7) (1997) 4619-4627.
\bibitem{FG} F. J. S. A. Corr$\hat{e}$a and  G. M. Figueiredo, {\it On an elliptic equation of $p$-Kirchhoff-type via variational methods,} Bull. Austral. Math. Soc. 77 (2006) 263-277.
 \bibitem{cor}   F. J. S. A. Corr$\hat{e}$a, {\it On positive solutions of nonlocal and nonvariational elliptic problems, Nonlinear Anal.} 59 (2004) 1147-1155.
\bibitem{fmr} D. G. de Figueiredo, O. H. Miyagaki, and B. Ruf, {\it Elliptic equations in $\mb R^2$
with nonlinearities in the critical growth range,} Calc. Var. Partial Differential
Equations 3 (2) (1995) 139-153.
\bibitem{DP} P. Drabek and S. I. Pohozaev, {\it Positive solutions for the
p-Laplacian: application of the fibering method}, Proc. Royal Soc.
Edinburgh Sect A 127 (1997) 703-726.
\bibitem{AE} A. El Hamidi, {\it Multiple solutions with changing
sign energy to a nonlinear elliptic equation}, Commun. Pure Appl.
Anal. 3 (2004) 253-265.
\bibitem{gs} G. M. Figueiredo, {\it Ground state soluttion for a Kirchhoff problem with exponential critical growth}, arXiv:1305.2571v1[math.AP].

\bibitem{JS} J. Giacomoni, and K. Sreenadh, {\it A multiplicity result to a nonhomogeneous elliptic equation in whole space $\Bbb R^2$,} Adv. Math. Sci. Appl. 15 (2) (2005) 467-488.

\bibitem{JSK} J. Giacomoni, S. Prashanth and K. Sreenadh, {\it A global
 multiplicity result for $N$-Laplacian with critical nonlinearity of
concave-convex type}, J. Differ. Equ. 232 (2007) 544-572.

\bibitem{HZ} X. He and W. Zou {\it Existence of a positive solution to Kirchhoff type problems without compactness conditions,}
J. Differ. Equ. 253 (7) (2012) 2285-2294.

\bibitem{XZ} X. He and W. Zou, {\it Existence and concentration behavior of positive solutions for a Kirchhoff equation in $\mb R^3$,} J. Differ.
Equ. 252 (2) (2012) 1813-1834.
\bibitem{LF} Y. Li, F. Li and J. Shi, {\it Infinitely many positive solutions for Kirchhoff-type problems}, Nonlinear Anal. 70 (2009) 1407-1414.

\bibitem{li} P. L. Lions, {\it The concentration compactness principle in the calculus of variations
part-I} Rev. Mat. Iberoamericana 1 (1985) 185-201.


\bibitem{do6} J. Marcos do \'{O}, {\it Semilinear Dirichlet problems for the $N$-Laplacian in
$\Om$ with nonlinearities in critical growth range}, Differential
Integral Equations 9 (1996) 967-979.

\bibitem{moser} J. Moser, {\it A sharp form of an inequality by N.
Trudinger}, Indiana Univ. Math. J. 20 (1971) 1077-1092.


\bibitem{PS} S. Prashanth and K. Sreenadh, {\it Multiplicity of Solutions to a nonhomogeneous elliptic equation
in $\mb R^2$}, Differential and Integral Equations 18 (2005)
681-698.


\bibitem{TA} G. Tarantello, {\it On nonhomogeneous elliptic equations involving critical Sobolev exponent},
Ann. Inst. H. Poincare- Anal. non lineaire 9 (1992) 281-304.



\bibitem{WU} T. F. Wu, {\it On semilinear elliptic equations involving concave-convex
nonlinearities and sign-changing weight function}, J. Math. Anal.
Appl., 318 (2006) 253-270.

\bibitem{WU9} T. F. Wu, {\it Multiplicity
results for a semilinear elliptic equation involving sign-changing
weight function}, Rocky Mountain J. Math. 39 (3) (2009) 995-1011.

\bibitem{WU10} T. F. Wu, {\it Multiple positive solutions for a class of
concave-convex elliptic problems in $\Om$ involving sign-changing
weight}, J. Funct. Anal. 258 (1) (2010) 99-131.




\end{thebibliography}
\end{document}